\documentclass[12pt, showkeys]{amsart}%
\usepackage{color,ulem}
\usepackage{graphicx}
\usepackage{subfigure}
\usepackage{amssymb}
\usepackage{amsmath,amsxtra}
\usepackage{multirow}
\usepackage{enumerate}
\usepackage{epstopdf}
\usepackage{cite,stmaryrd,txfonts}
\usepackage{tikz}
\usepackage{pgf}
\usetikzlibrary{snakes}
\usetikzlibrary{angles,calc,patterns}
\usetikzlibrary{shapes,arrows,automata}
\usepackage{bbm}
\usepackage{dsfont}
\usepackage{accents}
 \usepackage{float}

\usepackage{epic}
\usepackage{empheq}
\usepackage{cases}
\usepackage{diagbox}
\def\cequiv{\raisebox{-1.5mm}{$\;\stackrel{\raisebox{-3.9mm}{=}}{{\sim}}\;$}}

\newtheorem{theorem}{Theorem}[section]
\newtheorem{remark}[theorem]{Remark}

\newtheorem{lemma}[theorem]{Lemma}

\newtheorem{definition}[theorem]{Definition}

\newcounter{mnote}
\let\oldmarginpar\marginpar
\renewcommand\marginpar[1]{\-\oldmarginpar[\raggedleft\footnotesize #1]
  {\raggedright\footnotesize #1}}

\usepackage{footnote}
\usepackage{booktabs}

\usepackage{rotating}

\numberwithin{equation}{section}

\usepackage{cite}

\usepackage[colorlinks,linkcolor=black,
anchorcolor=black,
citecolor=black]{hyperref}
\usepackage[shortlabels]{enumitem}
\setlist[enumerate]{nosep}
\usepackage{color, soul}
\soulregister\cite7
\usepackage{setspace}

\def\dv{{\rm div}}

\def\oT{\mathbf{T}}

\def\vs{v^*}

\def\sX{\xX}
\def\sY{\yY}

\def\icr{{\sf icr}}

\def\R{\mathcal{R}}
\def\N{\mathcal{N}}

\usepackage[mathscr]{eucal}

\def\xD{\boldsymbol{\mathbf{D}}}

\def\xM{\boldsymbol{\mathbf{M}}}

\def\xX{{\boldsymbol{\mathbf{X}}}}
\def\xY{{\boldsymbol{\mathbf{Y}}}}

\def\xv{\boldsymbol{\mathbf{v}}}
\def\xw{\boldsymbol{\mathbf{w}}}

\def\yD{\boldsymbol{\mathbbm{D}}}

\def\yN{\boldsymbol{\mathbbm{N}}}

\def\yT{\boldsymbol{\mathbbm{T}}}

\def\yY{{\boldsymbol{\mathbbm{Y}}}}

\def\yv{\boldsymbol{\mathbbm{v}}}
\def\yw{\boldsymbol{\mathbbm{w}}}

\def\asD{\yD}

\def\tyN{\yN}

\def\sD{\xD}

\def\avv{\yv}
\def\avw{\yw}

\def\aoT{\yT}

\def\txM{\widetilde{\xM}}
\def\uxM{{\undertilde{\xM}}}

\def\tyN{\mathbf{\widetilde{{\mathbbm{N}}}}}
\def\uyN{{\undertilde{\yN}}}

\def\fzeta{\boldsymbol{\zeta}}
\def\feta{\boldsymbol{\eta}}

\def\bve{\boldsymbol{\varepsilon}}
\def\tr{{\rm tr}}
\def\Id{\boldsymbol{\bf Id}}

\def\mSigma{\boldsymbol{\Sigma}}

\def\msigma{\boldsymbol{\sigma}}
\def\mtau{\boldsymbol{\tau}}
\def\mzeta{\boldsymbol{\zeta}}
\def\meta{\boldsymbol{\eta}}

\def\mvarphi{\boldsymbol{\varphi}}
\def\mpsi{\boldsymbol{\psi}}

\def\vU{\boldsymbol{\rm U}}
\def\vV{\boldsymbol{\rm V}}

\def\vR{\boldsymbol{\rm R}}
\def\vP{\boldsymbol{\rm P}}

\def\vL{\boldsymbol{\rm L}}
\def\vH{\boldsymbol{\bf H}}

\def\vf{\boldsymbol{f}}

\def\vp{\boldsymbol{p}}
\def\vq{\boldsymbol{q}}
\def\vr{\boldsymbol{r}}
\def\vs{\boldsymbol{s}}

\def\vu{\boldsymbol{u}}
\def\vv{\boldsymbol{v}}
\def\vw{\boldsymbol{w}}

\begin{document}

\title[Low-degree robust symmetric schemes for planar linear elasticity]{Low-degree robust Hellinger-Reissner finite element schemes for planar linear elasticity with symmetric stress tensors}

\author{Shuo Zhang}
\address{LSEC, Institute of Computational Mathematics and Scientific/Engineering Computing, Academy of Mathematics and System Sciences, Chinese Academy of Sciences, Beijing 100190; University of Chinese Academy of Sciences, Beijing, 100049; People's Republic of China}
\email{szhang@lsec.cc.ac.cn}

\thanks{The research is partially supported by NSFC (11871465) and CAS (XDB41000000).}


%
%
%
%
%
%
%


\begin{abstract}
In this paper, we study the construction of low-degree robust finite element schemes for planar linear elasticity on general triangulations. 

Firstly, we present a low-degree nonconforming Helling-Reissner finite element scheme. For the stress tensor space, the piecewise polynomial shape function space is 
$$
{\rm span}\left\{\left(\begin{array}{cc}1&0\\ 0 & 0\end{array}\right),\left(\begin{array}{cc}0&1\\ 1 & 0\end{array}\right),\left(\begin{array}{cc}0&0\\ 0 & 1\end{array}\right),\left(\begin{array}{cc}0&x\\ x & 0\end{array}\right),\left(\begin{array}{cc}0&y\\ y & 0\end{array}\right),\left(\begin{array}{cc}0&x^2-y^2\\ x^2-y^2 & 0\end{array}\right)\right\},
$$
the dimension of the total space is asymptotically 8 times of the number of vertices, and the supports of the basis functions are each a patch of an edge. The piecewise rigid body space is used for the displacement. Robust error estimations in $\mathbb{L}^2$ and broken $\vH(\dv)$ norms are presented. 

Secondly, we investigate the theoretical construction of schemes with lowest-degree polynomial shape function spaces. Specifically, a Hellinger-Reissner finite element scheme is constructed, with the local shape function space for the stress tensor being 
$$
{\rm span}\left\{\left(\begin{array}{cc}1&0\\ 0 & 0\end{array}\right),\left(\begin{array}{cc}0&1\\ 1 & 0\end{array}\right),\left(\begin{array}{cc}0&0\\ 0 & 1\end{array}\right),\left(\begin{array}{cc}0&x\\ x & 0\end{array}\right),\left(\begin{array}{cc}0&y\\ y & 0\end{array}\right)\right\},
$$ 
which is of the lowest degree for the local approximation of $\vH(\dv;\mathbb{S})$, and the space for the displacement is piecewise constants. Robust error estimations in $\mathbb{L}^2$ and broken $\vH(\dv)$ norms are presented for regular solutions and data. Meanwhile, accompanied with this Hellinger-Reissner finite element scheme, a minimal Navier-Lam\'e finite element scheme is constructed, with the local shape function space for the displacement being 
$$
{\rm span}\left\{\left(\begin{array}{c}1\\ 0\end{array}\right),\ \left(\begin{array}{c}0\\ 1\end{array}\right),\ \left(\begin{array}{c}x\\ 0\end{array}\right),\ \left(\begin{array}{c}y\\ x\end{array}\right),\ \left(\begin{array}{c}0\\ y\end{array}\right)\right\},
$$
which is of lowest degree for the local approximation in the norm $\|\cdot\|_0+\|\bve(\cdot)\|_0$. Robust error estimations in broken {\it energy} norms are presented for regular solutions and data. 
\end{abstract}

\maketitle

\tableofcontents

\section{Introduction}

The Hellinger-Reissner variational principle of planar linear elasticity seeks simultaneously for a $2\times 2$ symmetric stress tensor $\msigma$ contained in $\vH(\dv,\Omega;\mathbb{S})$ which are square integrable with square integrable divergence and for a square integrable displacement vector $\vu$. The symmetry constraint imposed on the stress tensor distinguishes the problem from the mixed formulation of general second-order problems. Meanwhile, the Lam\'e constants, $\mu$ and $\lambda$, characterize certain mechanical feature of the material, and the issue that the computational performance keeps robust when $\lambda$ tends to infinity is one which is to be taken into consideration. The Hellinger-Reissner variational principle has been a long-standing foundational model problem of computational mechanics. 
~\\

For the Hellinger-Reissner finite element methods, the earlier works \cite{Watwood.V;Hartz.B1968equilibrium,Arnold.D;Douglas.J;Gupta.C1984,Veubeke.B1965displacement,Johnson.C;Mercier.B1978} mainly focuses on composite elements, namely a composite structure of the triangulation is assumed. Later on, kinds of schemes with weakly imposed symmetry are designed, cf., e.g., \cite{Amara.M;Thomas.J1979,Arnold.D;Brezzi.F;Douglas.J1984,Arnold.D;Falk.R1988,stenberg1986construction,Arnold.D;Falk.R;Winther.R2007mixed,Boffi.D;Brezzi.F;Fortin.M2009reduced,Cockburn.B;Gopalakrishnan.J;Guzman.J2010mc}. In 2002, based on the elasticity complex, the first family of symmetric conforming mixed elements with polynomial shape functions was proposed for the planar case in \cite{Arnold.D;Winther.R2002}, which was extended to the three dimension by \cite{Arnold.D;Awanou.G;Winther.R2008mc}. Those elements from \cite{Arnold.D;Winther.R2002} and \cite{Arnold.D;Awanou.G;Winther.R2008mc} are later extended to any dimension by a family of conforming elements, with fewer degrees of freedom, proposed in \cite{Hu.J2015jcm,Hu.J;Zhang.S2014family,Hu.J;Zhang.S2015scm,Hu.J;Zhang.S2016m3as} by discovering a crucial structure of symmetric matrix-valued polynomials on simplicial grids and proving two basic algebraic results. Besides, various conforming finite elements are also constructed in, e.g.,  \cite{Adams.S;Cockburn.B2005,Arnold.D;Awanou.G;Winther.R2008mc,Arnold.D;Awanou.G2005m3as,Hu.J2015sinum,Hu.J;Ma.R2018conforming,Hu.J;Man.H;Zhang.S2014jsc}, with or without an (explicit) application of the elasticity complex. For these  conforming finite elements, the robust stability can be proved once the divergence of the stress space is proved contained in the displacement space and an inf-sup condition is proved. The inequality for sufficiently smooth $2\times 2$ tensors, not necessarily symmetric, in the form that (see Lemma \ref{lem:anewpi} below from \cite{Boffi.D;Brezzi.F;Fortin.M2013} and discussed in \cite{Arnold.D;Douglas.J;Gupta.C1984,Boffi.D;Brezzi.F;Fortin.M2009reduced})
\begin{equation}\label{eq:cfnewpi}
\|\mtau\|_{0,\Omega}\leqslant C(\|\mtau^D\|_{0,\Omega}+\|\dv\mtau\|_{0,\Omega}),
\end{equation}
where $\mtau^D$ is the deviatoric part of $\mtau$, is crucial for the robust coercivity. To reduce the complexity of conforming finite element schemes, nonconforming finite element schemes are designed. We refer to \cite{Arnold.D;Awanou.G;Winther.R2014m3as,Awanou.G2012,Gopalakrishnan.J;Guzman.J2011sinum,hu2019nonconforming,Hu.J;Man.H;Wang.J;Zhang.S2016simplest,Man.H;Hu.J;Shi.Z2009lower,Yi.S2005calcolo,Arnold.D;Winther.R2003nonconforming,awanou2009rotated,Xie.X;Xu.J2011scm,Yi.S2006m3as,Hu.J;Shi.Z2008sinum} for some discussions. Due to the nonconformity, the discrete analogue of \eqref{eq:cfnewpi} does not hold self-evidently for the nonconforming tensor functions; though, a locking-free behavior of the scheme of \cite{hu2019nonconforming} can be illustrated by numerical experiments. Based on these existing works, for the planar elasticity, the dimension of local polynomial shape function space for the symmetric stress tensor can be as low as 7 on rectangular grids (\cite{Man.H;Hu.J;Shi.Z2009lower}) and 12 on triangular grids (\cite{Arnold.D;Winther.R2003nonconforming,Sun.Y;Chen.S;Yang.Y2019amc}) both  by nonconforming elements. As from the point of approximation theory, a shape function space with dimension as low as 5 can be minimally sufficient for the local approximation of $\vH(\dv;\mathbb{S})$, in this paper, we will try to narrow the gap. We will focus on the problem with pure displacement boundary condition; the other conditions can be dealt with the same way. Besides, we focus ourselves on two dimension, and the high-dimensional case will be discussed in future. 
~\\

In this paper, firstly, we present a low-degree nonconforming Helling-Reissner finite element scheme on general triangulations. For the stress tensor space $\mSigma^{\small\rm KS'}_h$, the 6-dimensional local polynomial shape function space is 
$$
{\rm span}\left\{\left(\begin{array}{cc}1&0\\ 0 & 0\end{array}\right),\left(\begin{array}{cc}0&1\\ 1 & 0\end{array}\right),\left(\begin{array}{cc}0&0\\ 0 & 1\end{array}\right),\left(\begin{array}{cc}0&x\\ x & 0\end{array}\right),\left(\begin{array}{cc}0&y\\ y & 0\end{array}\right),\left(\begin{array}{cc}0&y^2-x^2\\ y^2-x^2 & 0\end{array}\right)\right\},
$$
and the piecewise 3-dimensional rigid body space $\vU^{\vR}_h$ is used for the displacement. The newly constructed $\mSigma^{\small\rm KS'}_h$--$\vU^{\vR}_h$ scheme possesses a $\lambda$-robust convergence rate; particularly, it is proved that $\|\msigma-\msigma_h\|_{0,\Omega}+\|\vu-\vu_h\|_{0,\Omega}\leqslant Ch\|\vf\|_{0,\Omega}$ on convex domains with a constant $C$ independent of $\lambda$. 

The stress tensor space $\mSigma^{\small\rm KS'}_h$ is defined by reconstructing the adjoint relationship. To construct a finite element space by reconstructing certain adjoint relation is suggested in \cite{Zhang.S2022padao-arxiv}, where a family of lowest-degree nonconforming finite element discretizations of $H\Lambda^k$, the spaces of exterior differential forms in $\mathbb{R}^n$ with $0\leqslant k\leqslant n$, are constructed, which form a discretized de Rham complex and a commutative diagram. Revealed by the theory of partially adjoint operators developed in \cite{Zhang.S2022padao-arxiv}, the discrete adjoint relation can lead to kinds of basic properties. The methodology is then utilized in \cite{Zhang.S2022pfemdde,Zhang.S2022pfemhl}, where primal finite element schemes are designed for the $\mathbf {H}(\mathbf
  {d})\cap\mathbf {H}(\boldsymbol {\delta}) $ elliptic problem and the Hodge Laplace problem on $\Lambda^k$ in $\mathbb{R}^n$ for $1\leqslant k\leqslant n-1$, respectively. In this paper, the so called adjoint relation between $\vH(\dv,\Omega;\mathbb{S})$ and $(H^1_0(\Omega))^2$ reads
\begin{equation}\label{eq:acontiadjoint}
(\dv\mtau,\vv)+(\mtau,\bve(\vv))=0,\ \forall\,\mtau\in \vH(\dv,\Omega;\mathbb{S}),\ \vv\in (H^1_0(\Omega))^2. 
\end{equation}
Actually, we note here that the operator $\dv:\mathbb{L}^2(\Omega,\mathbb{S})\to \vL^2(\Omega)$ (see the notations in Section \ref{sec:pre}) with domain $\vH(\dv,\Omega;\mathbb{S})$, say $(\dv,\vH(\dv,\Omega;\mathbb{S}))$, and the operator $(\bve,(H^1_0(\Omega))^2):\vL^2(\Omega)\to \mathbb{L}^2(\Omega,\mathbb{S})$ are a pair of adjoint operators, and this is why the Navier-Lam\'e principle and the Hellinger-Reissner principle are mathematically equivalent. Then, denoting by $\vV^{\rm KS}_{h0}$ the Kouhia-Stenberg space for the Navier-Lam\'e variational principle\cite{kouhia1995linear}, we define $\mSigma^{\small\rm KS'}_h$ so that the discrete relationship holds as
\begin{equation}\label{eq:adisadjoint}
\sum_{T\in\mathcal{T}_h}[(\dv\mtau_h,\vv_h)_T+(\mtau_h,\bve(\vv_h))_T]=0,\ \forall\,\mtau_h\in \mSigma^{\small\rm KS'}_h,\ \vv_h\in \vV^{\rm KS}_{h0}.
\end{equation}
We use \eqref{eq:adisadjoint} as a criterion for $\mSigma^{\rm KS'}_h$, not only because it provides appropriate continuity restriction, but also because the adjoint relation \eqref{eq:acontiadjoint} is crucial for \eqref{eq:cfnewpi}, and the analogue \eqref{eq:adisadjoint} is much useful for reconstructing \eqref{eq:cfnewpi} with finite element functions and further the $\lambda$-independent coercivity. 
  
The $\mSigma^{\small\rm KS'}_h$ constructed this way does not coincide with a {\it finite element} triple, defined by \cite{Ciarlet.P1978book} and adopted in, e.g., \cite{Brenner.S;Scott.R2008,Shi.Z;Wang.M2013book} and other numerous textbooks. We simply call $\mSigma^{\rm KS'}_h$ one of {\it non-Ciarlet} type. The non-Ciarlet type finite element spaces/functions are defined by combining the local shape function spaces and global continuity restrictions, and on a cell there may not be a unisolvence. The analysis of these schemes may rely on indirect approaches. Though, once the locally supported basis functions are figured out, the scheme can be implemented by the subroutine for standard Ciarlet-type finite element methods. Famous non-Ciarlet type finite elements include the Fortin-Soulie element \cite{Fortin.M;Soulie.M1983}, the Park-Sheen element \cite{Park.C;Sheen.D2003} and others. Recently, some low(est)-degree finite element schemes for respective model problems are constructed in the form of non-Ciarlet type; we refer to \cite{Zhang.S2020IMA,Zhang.S2021SCM,Liu.W;Zhang.S2022jsc} for some examples., and to \cite{Xi.Y;Ji.X;Zhang.S2020jsc,Zeng.H;Zhang.C;Zhang.S2020arxiv,Xi.Y;Ji.X;Zhang.S2021cicp,Zeng.H;Zhang.C;Zhang.S2020+BIT} for some related applications. The aforementioned non-Ciarlet type finite element spaces of \cite{Zhang.S2022padao-arxiv,Zhang.S2022pfemdde,Zhang.S2022pfemhl} all admit locally supported basis functions as well. For the space $\mSigma^{\small\rm KS'}_h$ of this paper, we show that it admits a set of basis functions, the supports of which are each contained in a patch of two neighbored cells. The total number of basis functions is asymptotically 8 times of the number of vertices, lower than other known finite element spaces for $\msigma$ on a same triangulation. Any single cell is covered by the supports of at most 9 basis functions of $\mSigma^{\rm KS'}_h$ and 3 basis functions of $\vU^{\vR}_h$; namely, the size of piecewise stiffness matrix is not bigger than $12\times 12$. 
~\\

Based on the spaces $\mSigma^{\rm KS'}_h$ and $\vU^{\vR}_h$, secondly in this paper, we investigate the theoretical construction of schemes with lowest-degree polynomial shape function spaces. Specifically, a Hellinger-Reissner finite element scheme is constructed, with the local shape function space for the stress tensor being 
$$
{\rm span}\left\{\left(\begin{array}{cc}1&0\\ 0 & 0\end{array}\right),\left(\begin{array}{cc}0&1\\ 1 & 0\end{array}\right),\left(\begin{array}{cc}0&0\\ 0 & 1\end{array}\right),\left(\begin{array}{cc}0&x\\ x & 0\end{array}\right),\left(\begin{array}{cc}0&y\\ y & 0\end{array}\right)\right\},
$$ 
and the space for the displacement is piecewise constants. The space $\mSigma^{\rm rKS'}_h$ for stress tensor is simply a reduction of $\mSigma^{\rm KS'}_h$. Robust error estimations in $\mathbb{L}^2$ and broken $\vH(\dv)$ norms are presented for regular solutions and data. This may be viewed a positive confirmation that the theoretical lower bound of degree of polynomials is an active one. 

Meanwhile, accompanied with this Hellinger-Reissner finite element scheme, by constructing 
\begin{equation}
\sum_{T\in\mathcal{T}_h}[(\dv\mtau_h,\vv_h)+(\mtau_h,\bve(\vv_h))]=0,\ \forall\,\mtau_h\in \mSigma^{\small\rm rKS'}_h,\ \vv_h\in \vV^{\rm rKS}_{h0},
\end{equation}
we construct a lowest-degree space $\vV^{\rm rKS}_{h0}$ for the Navier-Lam\'e principle, with the local shape function space for the displacement being 
\begin{equation}\label{eq:lowestshapeNL}
{\rm span}\left\{\left(\begin{array}{c}1\\ 0\end{array}\right),\ \left(\begin{array}{c}0\\ 1\end{array}\right),\ \left(\begin{array}{c}x\\ 0\end{array}\right),\ \left(\begin{array}{c}y\\ x\end{array}\right),\ \left(\begin{array}{c}0\\ y\end{array}\right)\right\},
\end{equation}
which is of lowest degree for the local approximation in the norm $\|\cdot\|_0+\|\bve(\cdot)\|_0$. Robust error estimations in broken {\it energy} norms are presented for regular solutions and data. Note that the shape function space \eqref{eq:lowestshapeNL} does not contain complete vector linear polynomials, and it works for the approximation in norms of $\bve(\vu)$, but not in norms of $\nabla\vu$. 
~\\

The remaining of the paper is organized as follows. Section \ref{sec:pre} collects some preliminaries, including the model problem, some basic polynomial spaces and finite element spaces, and the theory of partially adjoint operators from \cite{Zhang.S2022padao-arxiv} Section \ref{sec:fes} introduces the finite element spaces $\mSigma^{\rm KS'}_h$, $\mSigma^{\rm rKS'}_h$ and $\vV^{\rm rKS}_{h0}$ and constructs the basic Poincar\'e inequalities thereon. Section \ref{sec:fescheme} studies the $\mSigma^{\rm KS'}_h-\vU^{\vR}_h$ Hellinger-Reissner finite element scheme, provides the robust error estimation, and presents the locally supported basis functions of $\mSigma^{\rm KS'}_h$. Section \ref{sec:lowest} studies the $\mSigma^{\rm rKS'}_h-\vU^m_h$ Hellinger-Reissner finite element scheme and the $\vV^{\rm rKS}_{h0}$ Navier-Lam\'e finite element scheme, which are respectively the lowest-degree one, and prove their robust convergence. 

\section{Preliminaries}
\label{sec:pre}

\subsection{Model problems}

In this paper, we use a Greek letter in bold format to denote a $2\times 2$ tensor, such as $\mtau=\left[\begin{array}{cc} \mtau_{11} & \mtau_{12} \\ \mtau_{21}&\mtau_{22}\end{array}\right]$, and use a Latin letter in bold format for a $2\times 1$ vector, such as $\vv=\left(\begin{array}{c}\vv^1 \\ \vv^2\end{array}\right)$. As usual, we use $\nabla$ and $\dv$ to denote the gradient operator and div operator, respectively. For a vector $\vv$, $\bve(\vv)=\frac{1}{2}(\nabla\vv+(\nabla\vv)^\top)$, where $\cdot^\top$ denotes the transpose. For a tensor, $\dv$ is performed row by row. We further denote by $\uuline{\mathbb{R}}{}_{\bf s}$ the space of symmetric $2\times 2$ tensors, by $\Id$ the identity $2\times 2$ tensor, and by $\tr(\cdot)$ the trace of a $2\times 2$ tensor.

Let $\Omega\subset\mathbb{R}^2$ be a polygon. In this paper, we consider the linear elasticity problem, with clamped boundary condition,
\begin{equation}\label{eq:pdemodel}
\left\{
\begin{array}{rll}
\mathcal{A}\msigma&=\bve(\vu)&\mbox{in}\,\Omega,
\\
\dv\msigma&=\vf&\mbox{in}\,\Omega,
\\
\vu&=\boldsymbol{0}&\mbox{on}\,\partial\Omega.
\end{array}
\right.
\end{equation}
Here $\mathcal{A}$ is the compliance tensor of fourth order, defined such that 
$$
\mathcal{A}_{ijkl}=\mathcal{A}_{klij}=\mathcal{A}_{jikl},\ \ 1\leqslant i,j,k,l\leqslant 2,
$$
and 
$$
\mathcal{A}\mtau=\frac{1}{2\mu}\left(\msigma-\frac{\lambda}{2\lambda+2\mu}(\tr \msigma)\Id\right), \ \ \mtau\in\uuline{\mathbb{R}}{}_{\bf s}.
$$
Here $\mu\in[\mu_0,\mu_1]$ and $0<\lambda<\infty$ are the Lam\'e constants. Let $\mathcal{C}$ denote the elasticity tensor, namely
$$
\mathcal{C}_{ijkl}=\mathcal{C}_{klij}=\mathcal{C}_{jikl},\ \ 1\leqslant i,j,k,l\leqslant 2,
$$
and
$$
\mtau=\mathcal{A}\mathcal{C}\mtau=\mathcal{C}\mathcal{A}\mtau,\ \ \mtau\in\uuline{\mathbb{R}}{}_{\bf s}.
$$
Then 
$$
\msigma=\mathcal{C}\bve(\vu)=\lambda\tr\bve(\vu)\Id+2\mu\bve(\vu).
$$

We use as usual $H^1(\Omega)$, $H^1_0(\Omega)$ and $L^2(\Omega)$ to denote certain Sobolev spaces, and specifically, denote $\displaystyle L^2_0(\Omega):=\left\{w\in L^2(\Omega):\int_\Omega w dx=0\right\}$. We use $\vL^2(\Omega):=(L^2(\Omega))^2$, and similarly $\vH^1(\Omega)$, $\vH^1_0(\Omega)$ and $\vH^2(\Omega)$, and $\mathbb{L}^2(\Omega;\mathbb{S}):=\{\mtau\in (\vL^2(\Omega))^2:\mtau_{12}=\mtau_{21}\}$.  We use $(\cdot,\cdot)$ to represent $L^2$ inner product of $L^2(\Omega)$, $\vL^2(\Omega)$ or $\mathbb{L}^2(\Omega;\mathbb{S})$ or on a subdomain of $\Omega$  with corresponding subscripts.

The Navier-Lam\'e variational principle of \eqref{eq:pdemodel} is: to find $\vu\in\vH^1_0(\Omega)$, such that 
\begin{equation}\label{eq:primalmodel}
(\mathcal{C}\bve(\vu),\bve(\vv))=2\mu(\bve(\vu),\bve(\vv))+\lambda(\dv\vu,\dv\vv)=(\vf,\vv),\ \ \forall\,\vv\in \vH^1_0(\Omega),
\end{equation}
and the Hellinger-Reissner variational principle of \eqref{eq:pdemodel} is: to find $\msigma\in \vH(\dv,\Omega;\mathbb{S})$ and $\vu\in \vL^2(\Omega)$, such that 
\begin{equation}\label{eq:model}
\left\{
\begin{array}{rlll}
(\mathcal{A}\msigma,\mtau)&+(\vu,\dv\mtau)&=0&\forall\,\mtau\in \vH(\dv,\Omega;\mathbb{S});
\\
(\dv\msigma,\vv)&&=(\vf,\vv)&\forall\,\vv\in \vL^2(\Omega).
\end{array}
\right.
\end{equation} 
The solution $\vu$ of \eqref{eq:primalmodel} and of \eqref{eq:model} are identical, and for \eqref{eq:model}, $\msigma=\mathcal{C}\bve(\vu)$.

\begin{theorem}
\cite{Brenner.S;Sung.L1992mc,Vogelius.M1983nm}
Let $\vu$ be the solution  of \eqref{eq:primalmodel}. There exists a constant $C$, independent of $\lambda$, such that 
$$
\lambda\|\dv\vu\|_\Omega\leqslant C\|\vf\|_{0,\Omega}.
$$
If $\Omega$ is convex, then further
$$
\|\vu\|_{2,\Omega}\leqslant C\|\vf\|_{0,\Omega}.
$$ 
\end{theorem}

Following \cite{Boffi.D;Brezzi.F;Fortin.M2009reduced,Boffi.D;Brezzi.F;Fortin.M2013}, we use $\mtau^D:=\mtau-\frac{1}{2}\tr (\mtau)\Id$ to denote the deviatoric part of $\mtau$, and then 
$$
(\msigma,\mtau)=\int_\Omega[\msigma^D:\mtau^D+\frac{1}{2}\tr(\msigma)\tr(\mtau)]
$$
and 
$$
(\mathcal{A}\msigma,\mtau)=\int_\Omega [\frac{1}{2\mu}\msigma^D:\mtau^D+\frac{1}{4(\lambda+\mu)}\tr (\msigma)\tr(\mtau)].
$$
\begin{lemma}\label{lem:anewpi}(\cite{Boffi.D;Brezzi.F;Fortin.M2013}, also \cite{Boffi.D;Brezzi.F;Fortin.M2009reduced,Arnold.D;Douglas.J;Gupta.C1984})
There exists a constant $C>0$, such that, for a $2\times 2$ tensor $\mtau\in \vH(\dv,\Omega)$ satisfying $\int_\Omega \tr(\mtau)dx=0$, we have 
$$
\|\mtau\|_{0,\Omega}\leqslant C(\|\mtau^D\|_{0,\Omega}+\|\dv\mtau\|_{0,\Omega}). 
$$
\end{lemma}

\subsection{Triangulations and basic finite element spaces}

Let $\mathcal{T}_h$ be a shape-regular triangular subdivision of $\Omega$ with mesh size $h$, such that $\overline\Omega=\cup_{T\in\mathcal{T}_h}\overline T$, and every boundary vertex is connected to at least one interior vertex. Denote by $\mathcal{E}_h$, $\mathcal{E}_h^i$, $\mathcal{E}_h^b$, $\mathcal{N}_h$, $\mathcal{N}_h^i$ and $\mathcal{N}_h^b$ the set of edges, interior edges, boundary edges, vertices, interior vertices, and boundary vertices, respectively. We use the subscript $``\cdot_h"$ to denote the dependence on triangulation. In particular, an operator with the subscript $``\cdot_h"$ indicates that the operation is performed cell by cell. 

We use the notations below for spaces of polynomials:
\begin{equation}
\mSigma^{\rm m}:={\rm span}\left\{\left(\begin{array}{cc}1&0\\ 0 & 0\end{array}\right),\left(\begin{array}{cc}0&1\\ 1 & 0\end{array}\right),\left(\begin{array}{cc}0&0\\ 0 & 1\end{array}\right),\left(\begin{array}{cc}0&x\\ x & 0\end{array}\right),\left(\begin{array}{cc}0&y\\ y & 0\end{array}\right)\right\},
\end{equation}

\begin{equation}
\mSigma^{\rm m+}:={\rm span}\left\{\left(\begin{array}{cc}1&0\\ 0 & 0\end{array}\right),\left(\begin{array}{cc}0&1\\ 1 & 0\end{array}\right),\left(\begin{array}{cc}0&0\\ 0 & 1\end{array}\right),\left(\begin{array}{cc}0&x\\ x & 0\end{array}\right),\left(\begin{array}{cc}0&y\\ y & 0\end{array}\right),\left(\begin{array}{cc}0&x^2-y^2\\ x^2-y^2 & 0\end{array}\right)\right\},
\end{equation}

\begin{equation}
\vP_1:={\rm span}\left\{\left(\begin{array}{c}1\\ 0\end{array}\right),\ \left(\begin{array}{c}0\\ 1\end{array}\right),\ \left(\begin{array}{c}x\\ 0\end{array}\right),\ \left(\begin{array}{c}y\\ 0\end{array}\right),\ \left(\begin{array}{c}0\\ x\end{array}\right),\ \left(\begin{array}{c}0\\ y\end{array}\right)\right\},
\end{equation}

\begin{equation}
\vV^{\bve,\rm m}:={\rm span}\left\{\left(\begin{array}{c}1\\ 0\end{array}\right),\ \left(\begin{array}{c}0\\ 1\end{array}\right),\ \left(\begin{array}{c}x\\ 0\end{array}\right),\ \left(\begin{array}{c}y\\ x\end{array}\right),\ \left(\begin{array}{c}0\\ y\end{array}\right)\right\},
\end{equation}

\begin{equation}
\vP_0:={\rm span}\left\{\left(\begin{array}{c}1\\ 0\end{array}\right),\ \left(\begin{array}{c}0\\ 1\end{array}\right)\right\},
\end{equation}
and
\begin{equation}
\vR:={\rm span}\left\{\left(\begin{array}{c}1\\ 0\end{array}\right),\ \left(\begin{array}{c}0\\ 1\end{array}\right),\ \left(\begin{array}{c}y\\ -x\end{array}\right)\right\}.
\end{equation}
Denote $\displaystyle\mSigma^{\rm m}(\mathcal{T}_h):=\{\mtau_h\in \mathbb{L}^2(\Omega):\mtau_h|_T\in \mSigma^{\rm m},\ \forall T\in\mathcal{T}_h\},$ and denote similarly $\displaystyle\mSigma^{\rm m+}(\mathcal{T}_h)$, $\displaystyle\vV^{\bve,\rm m}(\mathcal{T}_h)$, $\displaystyle\vP_1(\mathcal{T}_h)$, $\displaystyle\vP_0(\mathcal{T}_h)$ and $\displaystyle\vR(\mathcal{T}_h)$. Denote by $\mathbb{P}^{\vR}_T$ the $\vL^2$ projection on a cell $T$ to $\vR$, by $\mathbb{P}^{\vR}_h$ the $\vL^2$ projection on a cell $T$ to $\vR(\mathcal{T}_h)$, by $\mathbb{P}^0_T$ the $\vL^2$ projection on a cell $T$ to constant, by $\mathbb{P}^0_h$ the $\vL^2$ projection on a cell $T$ to piecewise constant. Particularly,  $\mathbb{P}^0_T$ and $\mathbb{P}^0_h$ are used for the projections to scalar, vector and tensor constants; no ambiguity will be come across based on the contexts.

\subsubsection{The Kouhia-Stenberg finite element space}

On a triangulation $\mathcal{T}_h$, denote the Kouhia-Stenberg finite element space
$$
\vV^{\rm KS}_{h0}:=\left\{\vv_h=(\vv_h^1,\vv_h^2)\in \vP_1(\mathcal{T}_h):\vv_h^1\in H^1_0(\Omega),\ \int_e\vv_h^2\ \mbox{is\ continuous\ across}\ e\in\mathcal{E}_h^i,\ \int_e\vv_h^2=0\ \mbox{on}\ e\in\mathcal{E}_h^b\right\}.
$$
The corresponding discretization problem for \eqref{eq:primalmodel} is: to find $\vu_h^{\rm KS}\in \vV^{\rm KS}_{h0}$, such that 
\begin{equation}\label{eq:primalks}
(\mathcal{C}\bve_h(\vu_h^{\rm KS}),\bve_h(\vv_h))=(\vf,\vv_h),\ \ \forall\,\vv_h\in \vV^{\rm KS}_{h0}. 
\end{equation}

\begin{lemma}\label{lem:kornks}\cite{kouhia1995linear}
There exists a constant $C$ depending on the shape regularity of $\mathcal{T}_h$, such that 
\begin{equation}
\inf_{q_h\in P_0(\mathcal{T}_h)\cap L^2_0(\Omega)}\sup_{\vv_h\in \vV^{\rm KS}_{h0}}\frac{(\dv_h\vv_h,q_h)}{\|\vv_h\|_{1,h}\|q_h\|_{0,\Omega}}\geqslant C,
\end{equation}
and
\begin{equation}\label{eq:weakkorn}
\|\nabla_h\vv{}_h\|_{0,\Omega}\leqslant C\|\bve_h(\vv{}_h)\|_{0,\Omega},\ \ \mbox{for}\ \ \vv_h\in \vV^{\rm KS}_{h0}.
\end{equation}
\end{lemma}

\begin{lemma}\label{lem:convks}\cite{kouhia1995linear}
Let $\vu$ and $\vu_h^{\rm KS}$ be the solutions of \eqref{eq:primalmodel} and \eqref{eq:primalks}, respectively. Denote $\msigma_h^{\rm KS}:=\mathcal{C}\bve_h(\vu_h^{\rm KS})$. If $\vu\in\vH^2(\Omega)$, then
$$
\|\vu-\vu_h^{\rm KS}\|_{1,h}+\|\msigma-\msigma_h^{\rm KS}\|_{0,\Omega}\leqslant Ch(\|\vu\|_{2,\Omega}+\|\vf\|_{0,\Omega}).
$$
\end{lemma}

\subsection{Theory of partially adjoint operators}

The theory of partially adjoint operators is introduced in \cite{Zhang.S2022padao-arxiv}, and we collect the basic results in this subsection. The details of the proof are omitted and referred to \cite{Zhang.S2022padao-arxiv} for simplicity. 

Let $\xX$ and $\yY$ be two Hilbert spaces with respective inner products $\langle\cdot,\cdot\rangle_{\xX}$ and $\langle\cdot,\cdot\rangle_{\yY}$, and respective norms $\|\cdot\|_{\xX}$ and $\|\cdot\|_{\yY}$. Let $(\oT,\txM):\xX\to \yY$ and $(\aoT,\tyN):\yY\to\xX$ be two closed operators, not necessarily densely defined. Denote, for $\vv\in\txM$, $\|\vv\|_\oT:=(\|\vv\|_\xX^2+\|\oT\vv\|_\yY^2)^{1/2}$, and for $\avv\in\tyN$, $\|\avv\|_{\aoT}:=(\|\avv\|_\yY^2+\|\aoT\avv\|_\xX^2)^{1/2}$. Denote
\begin{equation}\label{eq:uxm}
\uxM:=\left\{\vv\in \txM:\langle\vv,\aoT\avv\rangle_\xX-\langle\oT\vv,\avv\rangle_\yY=0,\ \forall\,\avv\in\tyN\right\},
\end{equation}
\begin{equation}\label{eq:uyn}
\uyN:=\left\{\avv\in \tyN:\langle\vv,\aoT\avv\rangle_\xX-\langle\oT\vv,\avv\rangle_\yY=0,\ \forall\,\vv\in\txM\right\},
\end{equation}
\begin{equation}\label{eq:xmb}
\xM_{\rm B}:=\left\{\vv\in \txM:\langle \vv,\vw \rangle_{\xX}=0,\forall\,\vw\in \N(\oT,\uxM);\ \langle\oT \vv,\oT \vw\rangle_{\yY}=0,\ \forall\,\vw\in \uxM\right\},
\end{equation}
and
\begin{equation}\label{eq:ynb}
\yN_{\rm B}:=\left\{\avv\in \tyN:\langle \avv,\avw \rangle_{\yY}=0,\forall\,\avw\in \N(\aoT,\uyN);\ \langle\aoT \avv,\aoT \avw\rangle_{\xX}=0,\ \forall\,\avw\in \uyN\right\}.
\end{equation}
We call $(\xM_{\rm B},\yN_{\rm B})$ the {\bf twisted} part of $(\txM,\tyN)$. 

In this paper, as usual, for an operator $(\oT,\sD)$, we use $\mathcal{N}(\oT,\sD):=\{\xv\in\sD:\oT\xv=0\}$ for the kernel space, and $\mathcal{R}(\oT,\sD):=\{\oT\xv:\xv\in\sD\}$ for the range space. 
\begin{definition}[Definition 2.13 of \cite{Zhang.S2022padao-arxiv}]\label{def:basepair}
A pair of closed operators $\left[(\oT,\txM):\xX\to\yY,(\aoT,\tyN):\yY\to\xX\right]$ is called a {\bf base operator pair}, if, with notations \eqref{eq:uxm}, \eqref{eq:uyn}, \eqref{eq:xmb} and \eqref{eq:ynb}, 
\begin{enumerate}
\item $\R(\oT,\txM)$, $\R(\aoT,\tyN)$, $\R(\oT,\uxM)$ and $\R(\aoT,\uyN)$ are all closed; 
\item $\mathcal{N}(\oT,\xM_{\rm B})$ and $\mathcal{R}(\aoT,\yN_{\rm B})$ are isomorphic, and $\mathcal{N}(\aoT,\yN_{\rm B})$ and $\mathcal{R}(\oT,\xM_{\rm B})$ are isomorphic. 
\end{enumerate}
\end{definition}

For $\left[(\oT,\txM),(\aoT,\tyN)\right]$ a base operator pair, for nontrivial $\R(\aoT,\yN_{\rm B})$ and $\N(\oT,\xM_{\rm B})$, denote 
\begin{equation}\label{eq:alpha}
\displaystyle \alpha:=\inf_{0\neq \vv\in\mathcal{N}(\oT,\xM_{\rm B})}\sup_{\vw\in \mathcal{R}(\aoT,\yN_{\rm B})}\frac{\langle\vv,\vw\rangle_\xX}{\|\vv\|_\xX\|\vw\|_\xX}
=
\inf_{0\neq \vw\in\mathcal{R}(\aoT,\yN_{\rm B})}\sup_{\vv\in\mathcal{N}(\oT,\xM_{\rm B})}\frac{\langle\vv,\vw\rangle_\xX}{\|\vv\|_\xX\|\vw\|_\xX},
\end{equation} 
and for nontrivial $\N(\aoT,\yN_{\rm B})$ and $\R(\oT,\xM_{\rm B})$, denote
\begin{equation}\label{eq:beta}
\beta:=\inf_{0\neq \avv\in\mathcal{N}(\aoT,\yN_{\rm B})}\sup_{\avw\in\mathcal{R}(\oT,\xM_{\rm B})}\frac{\langle\avv,\avw\rangle_\yY}{\|\avv\|_\yY\|\avw\|_\yY}
=
\inf_{0\neq \avw\in\mathcal{R}(\oT,\xM_{\rm B})}\sup_{\avv\in\mathcal{N}(\aoT,\yN_{\rm B})}\frac{\langle\avv,\avw\rangle_\yY}{\|\avv\|_\yY\|\avw\|_\yY}.
\end{equation}
Then $\alpha>0$ and $\beta>0$. We further make a convention that,
\begin{equation}\label{eq:alphabeta=1}
\displaystyle \left\{
\begin{array}{ll}
\alpha=1, & \mbox{if}\, \mathcal{N}(\oT,\xM_{\rm B})=\mathcal{R}(\aoT,\yN_{\rm B})=\left\{0\right\};
\\
\beta=1, & \mbox{if}\, \mathcal{N}(\aoT,\yN_{\rm B})=\mathcal{R}(\oT,\xM_{\rm B})=\left\{0\right\}.
\end{array}
\right.
\end{equation}

\begin{definition}[Definition 2.15 of \cite{Zhang.S2022padao-arxiv}]\label{def:pao}
For $\left[(\oT,\txM):\xX\to \yY,(\aoT,\tyN):\yY\to\xX\right]$ a base operator pair, two operators $(\oT,\sD)\subset(\oT,\txM)$ and $(\aoT,\asD)\subset(\aoT,\tyN)$ are called {\bf partially adjoint} based on $\left[(\oT,\txM),(\aoT,\tyN)\right]$, if
\begin{equation}\label{eq:pacondition}
\displaystyle \sD=\left\{\vv\in \txM:\langle \vv,\aoT \avv\rangle_\sX-\langle \oT \vv,\avv\rangle_\sY=0,\ \forall\,\avv\in\asD\right\},
\end{equation}
and
\begin{equation}
\displaystyle \asD=\left\{\avv\in \tyN: \langle \vv,\aoT \avv\rangle_\sX-\langle \oT \vv,\avv\rangle_\sY=0,\ \forall\,\vv\in \sD\right\}.
\end{equation}
\end{definition}

\begin{definition}[Definition 2.8 of \cite{Zhang.S2022padao-arxiv}]
For $(\oT,\xD):\xX\to \yY$ a closed  operator, denote 
$$
\xD^{\boldsymbol \lrcorner}:=\left\{\xv\in \xD:\langle \xv,\xw\rangle_\xX=0,\ \forall\,\xw\in \mathcal{N}(\oT,\xD)\right\}.
$$ Define the {\bf index of closed range} of $(\oT,\xD)$ as
\begin{equation}\label{eq:deficr}
\mathsf{icr}(\oT,\xD):=\left\{\begin{array}{rl}
\displaystyle \sup_{0\neq\xv\in \xD^{\boldsymbol \lrcorner}}\frac{\|\xv\|_\xX}{\|\oT\xv\|_\yY},&\mbox{if}\ \xD^{\boldsymbol \lrcorner}\neq\left\{0\right\};
\\
0,&\mbox{if}\ \xD^{\boldsymbol \lrcorner}=\left\{0\right\}.
\end{array}\right.
\end{equation}
\end{definition}
Note that $\icr(\oT,\sD)$ evaluates in $\left[0,+\infty\right]$, and $\R(\oT,\sD)$ is closed if and only if $\icr(\oT,\sD)<\infty$. Further, $\icr(\oT,\xD^{\boldsymbol \lrcorner})$ plays like the constant for Poincar\'e inequality in the sense that $\|\vv\|_\xX\leqslant \icr(\oT,\xD^{\boldsymbol \lrcorner})\|\oT\vv\|_\xY$ for $\vv\in\xD^{\boldsymbol \lrcorner}$.

\begin{theorem}[quantified closed range theorem, Theorem 2.21 of \cite{Zhang.S2022padao-arxiv}]\label{thm:chpi}
For $\left[(\oT,\sD),(\aoT,\asD)\right]$ partially adjoint based on $\left[(\oT,\txM),(\aoT,\tyN)\right]$, with notations given in \eqref{eq:uxm}, \eqref{eq:uyn}, \eqref{eq:alpha}, \eqref{eq:beta} and \eqref{eq:alphabeta=1}, if $\icr(\aoT,\asD)<\infty$,
\begin{equation}
\icr(\oT,\sD)\leqslant (1+\alpha^{-1})\cdot\icr(\oT,\txM)+\alpha^{-1}\icr(\aoT,\asD)+\icr(\oT,\uxM);
\end{equation}
if $\icr(\oT,\sD)<\infty$,
\begin{equation}\label{eq:icrtstardstar}
\icr(\aoT,\asD) \leqslant (1+\beta^{-1})\cdot\icr(\aoT,\tyN)+\beta^{-1}\icr(\oT,\sD)+\icr(\aoT,\uyN).
\end{equation}
\end{theorem}
\begin{remark}\label{rem:chpi}
In the context of finite-dimensional case, $\icr(\oT,\sD)$ and $\icr(\aoT,\asD)$ are always finite. If, for example, $\uxM=\{0\}$, then $\icr(\oT,\uxM)=0$; if further $\mathcal{N}(\oT,\xM_{\rm B})= \mathcal{R}(\aoT,\yN_{\rm B})$ thus $\alpha=1$, it holds that $\icr(\oT,\sD)\leqslant 2\icr(\oT,\txM)+\icr(\aoT,\asD)$.
\end{remark}

\section{Finite element spaces and adjoint properties}
\label{sec:fes}

On the subdivision $\mathcal{T}_h$, define 
\begin{equation}\label{eq:deffems}
\mSigma^{\rm KS'}_h:=\Big\{\mtau_h\in\mSigma^{\rm m+}(\mathcal{T}_h):(\dv\mtau_h,\vv_h)+(\mtau_h,\bve_h(\vv_h))=0,\ \forall\,\vv_h\in \vV^{\rm KS}_{h0}\Big\},
\end{equation}
\begin{equation}
\mSigma^{\rm rKS'}_h:=\{\mtau_h\in \mSigma^{\rm m}(\mathcal{T}_h):(\dv\mtau_h,\vv_h)+(\mtau_h,\bve_h(\vv_h))=0,\ \forall\,\vv_h\in \vV^{\rm KS}_{h0}\},
\end{equation}
and
$$
\vV^{\rm rKS}_{h0}:=\{\vv_h\in \vV^{\bve,\rm m}(\mathcal{T}_h):(\dv\mtau_h,\vv_h)+(\mtau_h,\bve_h(\vv_h))=0,\ \forall\,\vv_h\in \mSigma^{\rm rKS'}_h\}.
$$
Namely, $\mSigma^{\rm KS'}_h$ satisfies some adjoint relation accompanied with $\vV^{\rm KS}_{h0}$, $\mSigma^{\rm rKS'}_h$ is a reduced space of $\mSigma^{\rm KS'}_h$ in the degree of polynomials, and $\vV^{\rm rKS}_{h0}$ is reduced in the degree of polynomials and relaxed in the regularity from $\vV^{\rm KS}_{h0}$. In this section, we construct their associated Poincar\'e inequalities by the theory of partially adjoint operators. 

\subsection{Base operator pairs for discretization}

\begin{lemma}
The pair $\Big[(\dv_h,\mSigma^{\rm m+}(\mathcal{T}_h)):\mathbb{L}^2(\Omega,\mathbb{S})\to \vL^2(\Omega), (\bve_h,\vP_1(\mathcal{T}_h)):\vL^2(\Omega)\to \mathbb{L}^2(\Omega,\mathbb{S})\Big]$ is a base operator pair. 
\end{lemma}
\begin{proof}
Evidently,
$$
\N(\dv,\mSigma^{\rm m+}(T))={\rm span}\left\{\left(\begin{array}{cc}1&0\\ 0 & 0\end{array}\right),\left(\begin{array}{cc}0&1\\ 1 & 0\end{array}\right),\left(\begin{array}{cc}0&0\\ 0 & 1\end{array}\right)\right\}=\R(\bve,\vP_1(T)),
$$
and 
$$
\R(\dv,\mSigma^{\rm m+}(T))={\rm span}\left\{\left(\begin{array}{c}1 \\ 0\end{array}\right),\left(\begin{array}{c}0 \\ 1\end{array}\right),\left(\begin{array}{c}y \\ -x\end{array}\right)\right\}=\N(\bve,\vP_1(T)).
$$
Now, given $\mtau\in \mSigma^{\rm m+}(T)$, $\mtau=0$ if and only if 
$$
(\dv\mtau,\vv)_T+(\mtau,\varepsilon(\vv))_T=0,\ \forall\,\vv\in \vP_1(T). 
$$
Therefore, given $\mtau_h\in \mSigma^{\rm m+}(\mathcal{T}_h)$, $\mtau_h=0$ if and only if 
$$
(\dv\mtau_h,\vv_h)+(\mtau_h,\bve_h(\vv_h))=0,\ \forall\,\vv_h\in \vP_1(\mathcal{T}_h). 
$$

Similarly, given $\vv_h\in \vP_1(\mathcal{T}_h)$, $\vv_h=0$ if and only if
$$
(\dv\mtau_h,\vv_h)+(\mtau_h,\bve_h(\vv_h))=0,\ \forall\,\mtau_h\in \mSigma^{\rm m+}(\mathcal{T}_h). 
$$
Therefore, the twisted part of $(\mSigma^{\rm m+}(\mathcal{T}_h),\vP_1(\mathcal{T}_h))$ is the pair itself. 

It is easy to obtain:
\begin{equation}\label{eq:ndivm+=rbvep1}
\N(\dv_h,\mSigma^{\rm m+}(\mathcal{T}_h))=\prod_{T\in\mathcal{T}_h}\N(\dv,\mSigma^{\rm m+}(T))=\prod_{T\in\mathcal{T}_h}\R(\bve,\vP_1(T))=\R(\bve_h,\vP_1(\mathcal{T}_h)),
\end{equation}
and similarly
\begin{equation}\label{eq:rdivm+=nbvep1}
\R(\dv_h,\mSigma^{\rm m+}(\mathcal{T}_h))=\N(\bve_h,\vP_1(\mathcal{T}_h)).
\end{equation}

Therefore $\Big[(\dv_h,\mSigma^{\rm m+}(\mathcal{T}_h)):\mathbb{L}^2(\Omega,\mathbb{S})\to \vL^2(\Omega), (\bve_h,\vP_1(\mathcal{T}_h)):\vL^2(\Omega)\to \mathbb{L}^2(\Omega,\mathbb{S})\Big]$ is a base operator pair by Definition \ref{def:basepair}. The proof is completed. 
\end{proof}

\begin{remark}
$\dim(\mSigma^{\rm KS'}_h)=\dim(\mSigma^{\rm m+}(\mathcal{T}_h))-\dim(\vV^{\rm KS}_{h0})=6\#(\mathcal{T}_h)-\#(\mathcal{E}_h^i)-\#(\mathcal{N}_h^i)=8\#(\mathcal{N}_h)+o(\#(\mathcal{N}_h))$.
\end{remark}

\begin{lemma}\label{lem:pidivpa}
For $\left[(\dv_h,\sD_h),(\bve_h,\asD_h)\right]$ partially adjoint based on $\Big[(\dv_h,\mSigma^{\rm m+}(\mathcal{T}_h)), (\bve_h,\vP_1(\mathcal{T}_h))\Big]$, 
\begin{equation}
\icr(\dv_h,\sD_h)\leqslant 2\cdot\icr(\dv_h,\mSigma^{\rm m+}(\mathcal{T}_h))+\icr(\bve_h,\asD_h).
\end{equation}
\end{lemma}
\begin{proof}
The lemma can be proved by noting $(\oT,\txM)=(\dv_h,\mSigma^{\rm m+}(\mathcal{T}_h))$, $(\aoT,\tyN)=(\bve_h,\vP_1(\mathcal{T}_h))$ and $\uxM=\{0\}$ in Theorem \ref{thm:chpi} with $\alpha=1$; see also Remark \ref{rem:chpi}. 
\end{proof}

\begin{lemma}\label{lem:baseicr}
There is a constant $C$ depending on the shape regularity of $\mathcal{G}_h$ only, such that 
$$
\icr(\dv_h,\mSigma^{\rm m+}(\mathcal{T}_h))\leqslant C.
$$
If further $\mathcal{G}_h$ is quasi-uniform,  
$$
\icr(\dv_h,\mSigma^{\rm m+}(\mathcal{T}_h))\leqslant Ch.
$$
\end{lemma}
\begin{proof}
For $\mSigma^{\rm m+}(T)$ on a cell $T$,  $\N(\dv,\mSigma^{\rm m+}(T))={\rm span}\left\{\left(\begin{array}{cc}1&0\\ 0 & 0\end{array}\right),\left(\begin{array}{cc}0&1\\ 1 & 0\end{array}\right),\left(\begin{array}{cc}0&0\\ 0 & 1\end{array}\right)\right\}$. Therefore, if $\mtau\in \mSigma^{\rm m+}(T)$ and $\mtau$ is orthogonal to $\N(\dv,\mSigma^{\rm m+}(T))$, namely $\mtau\in \mSigma^{\rm m+}(T)^\lrcorner$, by elementary calculation, $\|\mtau\|_{0,T}\leqslant Ch_T\|\dv\mtau\|_{0,T}$. Namely $\icr(\dv,\mSigma^{\rm m+}(T))\leqslant Ch_T$. 

Noting that $\displaystyle\N(\dv_h,\mSigma^{\rm m+}(\mathcal{T}_h))=\prod_{T\in\mathcal{T}_h}\N(\dv,\mSigma^{\rm m+}(T))$, and $\displaystyle\mSigma^{\rm m+}(\mathcal{T}_h)^\lrcorner=\prod_{T\in\mathcal{T}_h}\mSigma^{\rm m+}(T)^\lrcorner$, we can show $\displaystyle\icr(\dv_h,\mSigma^{\rm m+}(\mathcal{T}_h))\leqslant \max_{T\in\mathcal{T}_h}\icr(\dv,\mSigma^{\rm m+}(T))$. The assertions follow by definition. 
\end{proof}

Evidently, 
\begin{equation}\label{eq:rker}
\N(\dv,\mSigma^{\rm m}(T))={\rm span}\left\{\left(\begin{array}{cc}1&0\\ 0 & 0\end{array}\right),\left(\begin{array}{cc}0&1\\ 1 & 0\end{array}\right),\left(\begin{array}{cc}0&0\\ 0 & 1\end{array}\right)\right\}=\N(\bve,\vV^{\rm m}(T)),
\end{equation} 
and 
\begin{equation}\label{eq:kerr}
\R(\dv,\mSigma^{\rm m}(T))=\vP_0(T)=\N(\bve,\vV^{\rm m}(T)).
\end{equation} 
Therefore, the lemma below follows easily.
\begin{lemma}
\begin{enumerate}
\item $[(\dv_h,\mSigma^{\rm m}(\mathcal{T}_h)), (\bve_h,\vV^{\rm m}(\mathcal{T}_h))]$ is a base operator pair;
\item for $\left[(\dv_h,\sD_h),(\bve_h,\asD_h)\right]$ partially adjoint based on $\Big[(\dv_h,\mSigma^{\rm m}(\mathcal{T}_h)), (\bve_h,\vV^{\bve,\rm m}(\mathcal{T}_h))\Big]$,  
$$
\icr(\bve_h,\asD_h)\leqslant 2\cdot\icr(\bve_h,\vV^{\bve,\rm m}(\mathcal{T}_h))+\icr(\dv_h,\sD_h);
$$
\item there is a constant $C$ depending on the shape regularity of $\mathcal{G}_h$ only, such that 
$$
\icr(\bve_h,\vV^{\bve,\rm m}(\mathcal{T}_h))\leqslant C;
$$
if further $\mathcal{G}_h$ is quasi-uniform,  
$$
\icr(\bve_h,\vV^{\bve,\rm m}(\mathcal{T}_h))\leqslant Ch.
$$
\end{enumerate}
\end{lemma}

\subsection{Poincar\'e inequality for $\mSigma^{\rm KS'}_h$}

By \eqref{eq:ndivm+=rbvep1} and \eqref{eq:rdivm+=nbvep1}, it follows that
\begin{equation}\label{eq:femsdual}
\vV^{\rm KS}_{h0}=\Big\{\vv_h\in\vP_1(\mathcal{T}_h):(\dv\mtau_h,\vv_h)+(\mtau_h,\bve_h(\vv_h))=0,\ \forall\,\mtau_h\in \mSigma^{\rm KS'}_h\Big\}.
\end{equation}
Therefore, by Definition \ref{def:pao}, the pair $\left[(\dv_h,\mSigma^{\rm KS'}_h),(\bve_h,\vV^{\rm KS}_{h0})\right]$ is {\it partially adjoint} based on the pair $\Big[(\dv_h,\mSigma^{\rm m+}(\mathcal{T}_h)), (\bve_h,\vP_1(\mathcal{T}_h))\Big]$.

\begin{lemma}\label{lem:anotherpi}
There exists a constant $C>0$, such that, for $\mtau_h\in \mSigma_h^{\rm KS'}$ satisfying $\int_\Omega \tr(\mtau_h)dx=0$, we have 
$$
\|\mtau_h\|_{0,\Omega}\leqslant C(\|\mtau_h^D\|_{0,\Omega}+\|\dv\mtau_h\|_{0,\Omega}). 
$$
\end{lemma}
\begin{proof}
Note that $\tr(\mtau_h)$ is a piecewise constant function. As $\int_\Omega \tr(\mtau_h)dx=0$, by Lemma \ref{lem:kornks}, there exists a $\vv_h\in \vV^{\rm KS}_{h0}$, such that 
$$
\dv_h\vv_h=\tr(\mtau_h),\ \ \mbox{and}\ \ \|\vv{}_h\|_{1,h}\leqslant C\|\tr(\mtau_h)\|_{0,\Omega}.
$$
Therefore,
\begin{multline*}
\|\tr(\mtau_h)\|_{0,\Omega}^2=(\tr(\mtau_h),\dv\vv_h)=\int_\Omega \mtau_h:\Id\,\dv_h\vv_h=2\int_\Omega \mtau_h:(\nabla_h \vv_h-(\nabla_h\vv_h)^D)
\\
=2\int_\Omega \mtau_h:(\bve_h(\vv_h)-(\nabla_h\vv_h)^D)=-2\int_\Omega \dv\mtau_h\vv_h-2\int_\Omega\mtau_h^D:\nabla_h\vv_h
\\
\leqslant 2(\|\dv_h\mtau_h\|_{0,\Omega}+\|\mtau_h^D\|_{0,\Omega})(\|\vv_h\|_{0,\Omega}+\|\nabla_h\vv_h\|_{0,\Omega}).
\end{multline*}
It follows that $\|\tr(\mtau_h)\|_{0,\Omega}\leqslant C (\|\dv_h\mtau_h\|_{0,\Omega}+\|\mtau_h^D\|_{0,\Omega})$. Therefore,
$$
\|\mtau_h\|_{0,\Omega}\leqslant \|\mtau_h^D\|_{0,\Omega}+\frac{1}{2}\|\tr(\mtau_h)\Id\|_{0,\Omega}\leqslant C (\|\dv_h\mtau_h\|_{0,\Omega}+\|\mtau_h^D\|_{0,\Omega}). 
$$
The proof is completed. 
\end{proof}

\begin{remark}
The adjoint relation \eqref{eq:adisadjoint} is crucial in the proof of Lemma \ref{lem:anotherpi}. 
\end{remark}

Denote 
$$
\mathring{\mSigma}^{\rm KS'}_h:=\N(\dv_h,\mSigma^{\rm KS'}_h),\ \mbox{and}\  \mSigma^{\rm KS',\lrcorner}_h:=\{\mtau_h\in \mSigma^{\rm KS'}_h:(\mathcal{A}\mtau_h,\feta_h)=0,\ \forall\,\feta_h\in \mathring{\mSigma}^{\rm KS'}_h\}.
$$
By Lemma \ref{lem:anotherpi}, if $\mtau_h\in \mathring{\mSigma}^{\rm KS'}_h$ and $\int_\Omega\tr(\mtau_h)=0$, then $\|\mtau_h\|_{0,\Omega}\leqslant C\|\mtau_h\|_{\mathcal{A}}:=C\sqrt{(\mathcal{A}\mtau_h,\mtau_h)}$ uniformly for $0<\lambda<\infty$.

\begin{lemma}\label{lem:pifes}
It holds with a constant $C$ depending on the regularity of $\mathcal{T}_h$ that 
\begin{equation}
\|\mtau_h\|_{0,\Omega}\leqslant C\|\dv_h\mtau_h\|_{0,\Omega},\ \forall\,\mtau_h\in \mSigma^{\rm KS',\lrcorner}_h. 
\end{equation}
\end{lemma}
\begin{proof}
Denote $\widetilde\mSigma^{\rm KS',\lrcorner}_h:=\{\mtau_h\in \mSigma^{\rm KS'}_h:(\mtau_h,\feta_h)=0,\ \forall\,\feta_h\in \mathring{\mSigma}^{\rm KS'}_h\}$. By Lemma \ref{lem:kornks}, $\icr(\bve_h,\vV^{\rm KS}_{h0})\leqslant C$. Then by Lemmas \ref{lem:baseicr} and \ref{lem:pidivpa}, $\icr(\dv_h,\mSigma_h^{\rm KS'}) \leqslant C$, namely $\|\mtau_h\|_{0,\Omega}\leqslant C\|\dv_h\mtau_h\|_{0,\Omega}$, for $\mtau_h\in \widetilde\mSigma^{\rm KS',\lrcorner}_h$. 

Now, given $\mtau_h\in \mSigma^{\rm KS',\lrcorner}_h$, there is a $\mtau_h'\in \widetilde\mSigma^{\rm KS',\lrcorner}_h$, such that $\dv_h\mtau_h=\dv_h\mtau_h'$, and $(\mathcal{A}\mtau_h',\meta_h)=(\mathcal{A}(\mtau_h'-\mtau_h),\meta_h)$ for any $\meta_h\in\mathring{\mSigma}^{\rm KS'}_h$. As $\mtau_h'-\mtau_h\in \mathring{\mSigma}^{\rm KS'}_h$ and $\int_\Omega \tr(\mtau_h'-\mtau_h)=0$, 
$$
\|\mtau_h'-\mtau_h\|_{0,\Omega}\leqslant C\|\mtau_h'-\mtau_h\|_{\mathcal{A}}\leqslant \|\mtau_h'\|_{\mathcal{A}}\leqslant C\|\mtau_h'\|_{0,\Omega}.
$$
Therefore, $\|\mtau_h\|_{0,\Omega}\leqslant C\|\mtau_h'\|_{0,\Omega}\leqslant C\|\dv_h\mtau_h'\|_{0,\Omega}= C\|\dv_h\mtau_h\|_{0,\Omega}$. The proof is completed.  
\end{proof}

\subsection{Poincar\'e inequality for reduced finite element spaces}

\begin{lemma}
$\mSigma^{\rm rKS'}_h=\{\mtau_h\in \mSigma^{\rm KS'}_h:\dv_h\mtau_h\in\vP_0(\mathcal{T}_h)\}$.
\end{lemma}
\begin{proof}
Evidently, $\mSigma^{\rm rKS'}_h\subset\{\mtau_h\in \mSigma^{\rm KS'}_h:\dv_h\mtau_h\in\vP_0(\mathcal{T}_h)\}$. On the other hand, given $\mtau_h\in \mSigma^{\rm KS'}_h$ such that $\dv_h\mtau_h\in\vP_0(\mathcal{T}_h)$, $\mtau_h\in \mSigma^{\rm m}(\mathcal{T}_h)$ and $(\dv\mtau_h,\vv_h)+(\mtau_h,\bve_h(\vv_h))=0$, $\forall\,\vv_h\in \vV^{\rm KS}_{h0}$; namely $\mtau_h\in \mSigma^{\rm rKS'}_h$. The proof is completed. 
\end{proof}
\begin{remark}
By the proof of Lemma \ref{lem:equiv}(see below), $\R(\dv_h,\mSigma^{\rm KS'}_h)=\vR_h(\mathcal{T}_h)$. Thus $\R(\dv_h,\mSigma^{\rm rKS'}_h)=\vP_0(\mathcal{T}_h)$ by its definition. 
\end{remark}
Note that $\N(\dv_h,\mSigma^{\rm rKS'}_h)=\N(\dv_h,\mSigma^{\rm KS'}_h)$, it follows that $\icr(\dv_h,\mSigma^{\rm rKS'}_h)\leqslant \icr(\dv_h,\mSigma^{\rm KS'}_h)\leqslant C$. Denote $\mSigma^{\rm rKS',\lrcorner}_h:=\{\mtau_h\in \mSigma^{\rm rKS'}_h:(\mtau_h,\meta_h)=0,\ \forall\,\meta_h\in \mathring{\mSigma}^{\rm KS'}_h\}.$ 
\begin{lemma}
There is a constant $C$ depending on the shape regularity of $\mathcal{G}_h$, such that 
$$
\|\mtau_h\|_{0,\Omega}\leqslant C\|\dv_h\mtau_h\|_{0,\Omega},\ \ \mtau_h\in \mSigma^{\rm rKS',\lrcorner}_h. 
$$
\end{lemma}

By \eqref{eq:rker} and \eqref{eq:kerr}, 
$$
\mSigma^{\rm rKS'}_h=\{\mtau_h\in \mSigma^{\rm m}(\mathcal{T}_h):(\dv\mtau_h,\vv_h)+(\mtau_h,\bve_h(\vv_h))=0,\ \forall\,\vv_h\in \vV^{\rm rKS}_{h0}\}.
$$
Then $\left[(\dv_h,\mSigma^{\rm rKS'}_h),(\bve_h,\vV^{\rm rKS}_{h0})\right]$ is partially adjoint based on $\Big[(\dv_h,\mSigma^{\rm m+}(\mathcal{T}_h)), (\bve_h,\vP_1(\mathcal{T}_h))\Big]$. 

\begin{lemma}
$\|\vv_h\|_{0,\Omega}\leqslant C\|\bve_h(\vv_h)\|_{0,\Omega}$, $\vv_h\in \vV^{\rm rKS}_{h0}$. 
\end{lemma}
\begin{proof}
Given $\vv_h\in \vV^{\rm rKS}_{h0}$ such that $\bve_h(\vv_h)=0$, then $\vv_h\in\vP_0(\mathcal{T}_h)$, and further $(\vv_h,\dv_h\mtau_h)=0,\ \ \forall\,\mtau_h\in \mSigma^{\rm rKS'}_h.$ Therefore, $\vv_h=\boldsymbol{0}$ as $\R(\dv_h,\mSigma^{\rm rKS'}_h)=\vP_0(\mathcal{T}_h)$. 

Note that $[(\dv_h,\mSigma^{\rm rKS'}_h), (\bve_h,\vV^{\rm rKS}_{h0})]$ is partially adjoint based on $[(\dv_h,\mSigma^{\rm m}(\mathcal{T}_h)), (\bve_h,\vV^{\bve,\rm m}(\mathcal{T}_h))]$. It follows that $\icr(\bve_h,\vV^{\rm rKS}_{h0})\leqslant C$; namely, $\|\vv_h\|_{0,\Omega}\leqslant C\|\bve_h(\vv_h)\|_{0,\Omega}$, $\vv_h\in \vV^{\rm rKS}_{h0}$. 
\end{proof}

\section{A low-degree Hellinger-Reissner finite element scheme}
\label{sec:fescheme}

We consider the finite element problem: find $\msigma_h\in \mSigma^{\rm KS'}_h$ and $\vu_h\in \vU^{\vR}_h:=\vR(\mathcal{T}_h)$, such that 
\begin{equation}\label{eq:fes}
\left\{
\begin{array}{rlll}
(\mathcal{A}\msigma_h,\mtau_h)&+(\vu_h,\dv\mtau)&=0&\forall\,\mtau_h\in \mSigma^{\rm KS'}_h,
\\
(\dv\msigma_h,\vv_h)&&=(\vf,\vv_h)&\forall\,\vv_h\in \vU^{\vR}_h. 
\end{array}
\right.
\end{equation}

The main results of this paper are the theorems below. 

\begin{theorem}\label{thm:wellpose}
Given $\vf\in \vL^2(\Omega)$, the system \eqref{eq:fes} admits a unique solution $(\msigma_h,\vu_h)$, and 
$$
\|\msigma_h\|_{\dv_h}+\|\vu_h\|_{0,\Omega}\cequiv \|\mathbb{P}^{\vR}_h\vf\|_{0,\Omega}. 
$$
The equivalence is independent of $\lambda$. 
\end{theorem}

\begin{theorem}\label{thm:errest}
Let $(\msigma,\vu)$ and $(\msigma_h,\vu_h)$ be the solutions of \eqref{eq:model} and \eqref{eq:fes}, respectively. Assume $\vu\in \vH^2(\Omega)$. Then, with $C$ independent of $\lambda$,
$$
\|\vu-\vu_h\|_{0,\Omega}\leqslant  Ch(\|\vu\|_{2,\Omega}+\|\vf\|_{0,\Omega}),
$$
$$
\|\msigma-\msigma_h\|_{0,\Omega}\leqslant  Ch(\|\vu\|_{2,\Omega}+\|\vf\|_{0,\Omega}),
$$
and
$$
\|\dv\msigma-\dv\msigma_h\|_{0,\Omega}\leqslant Ch\|\vf\|_{1,\Omega},\ \ \mbox{provided\ that}\ \vf\in \vH^1(\Omega). 
$$
\end{theorem}
We postpone the proofs of the two theorems after some technical preparations. 
\begin{remark}
If $\Omega$ is convex, then
$$
\|\vu-\vu_h\|_{0,\Omega}+\|\msigma-\msigma_h\|_{0,\Omega}\leqslant  Ch\|\vf\|_{0,\Omega}.
$$
\end{remark}

\subsection{Equivalence between finite element systems}

\begin{lemma}\label{lem:equiv}
Given $\vf_h\in \vU^{\vR}_h$. Let $\vr_h\in \vV^{\rm KS}_{h0}$ solve 
\begin{equation}\label{eq:ksred}
(\mathcal{C}\bve_h(\vr_h),\bve_h(\vs_h))=(\vf_h,\vs_h),\ \forall\,\vs_h\in \vV^{\rm KS}_{h0},
\end{equation}
and let $(\fzeta_h,\bar{\vr}_h)\in\mSigma^{\rm KS'}_h\times \vU^{\vR}_h$ solve 
\begin{equation}\label{eq:stressred}
\left\{
\begin{array}{rlll}
(\mathcal{A}\mathbb{P}^0_h\fzeta_h,\mathbb{P}^0_h \feta_h)&-(\bar{\vr}_h,\dv_h\feta_h)&=0, & \forall\,\feta_h\in \mSigma_h^{\rm KS'}
\\
(\bar{\vs}_h,\dv_h\fzeta_h)&&=(\vf_h,\bar{\vs}_h)&\forall\,\bar{\vs}_h\in \vU^{\vR}_h.
\end{array}
\right.
\end{equation}
Then
\begin{equation}
\bar{\vr}_h=\mathbb{P}^0_h\vr_h,\ \ \mathbb{P}^0_h\mzeta_h=\mathcal{C}\bve_h(\vr_h),\ \ \mbox{and}\ \  \dv_h\mzeta_h=\vf_h. 
\end{equation}
\end{lemma}

\begin{proof}
The solution of \eqref{eq:stressred}, if exists, is unique. Actually, if $(\fzeta_h,\bar{\vr}_h)\in\mSigma^{\rm KS'}_h\times \vU^{\vR}_h$ solve 
\begin{equation}
\left\{
\begin{array}{rlll}
(\mathcal{A}\mathbb{P}^0_h\fzeta_h,\mathbb{P}^0_h \feta_h)&-(\bar{\vr}_h,\dv_h\feta_h)&=0, & \forall\,\feta_h\in \mSigma_h^{\rm KS'}
\\
(\bar{\vs}_h,\dv_h\fzeta_h)&&=0&\forall\,\bar{\vs}_h\in \vU^{\vR}_h,
\end{array}
\right.
\end{equation}
then, by decomposing $\mzeta_h=\mathring{\mzeta}_h+\mzeta_h^\lrcorner$ with $\mathring{\mzeta}_h\in \mathring{\mSigma}^{\rm KS'}_h$ and $\mzeta_h^\lrcorner\in \mSigma^{\rm KS',\lrcorner}_h$, we can show that $\mzeta_h=0$ by noting easily that $\mathring{\mzeta}_h=0$ and $\mzeta_h^\lrcorner=0$. Further, $\bar{\vr}_h=0$.

The solution $\vr_h$ of \eqref{eq:ksred} can be viewed as, $\vr_h\in\vP_1(\mathcal{T}_h)$, such that, with $\fzeta_h\in \mSigma^{\rm KS'}_h$,
\begin{equation}\label{eq:kssaddle}
\left\{
\begin{array}{rlll}
(\mathcal{C}\bve_h(\vr_h),\bve_h(\vs_h)) &+(\fzeta_h,\bve_h(\vs_h))+(\dv_h\fzeta_h,\vs_h)&=(\vf_h,\vs_h),& \forall\,\vs_h\in \vP_1(\mathcal{T}_h)
\\
(\bve_h(\vr_h),\feta_h)+(\vr_h,\dv_h\feta_h)&&=0,& \forall\,\feta_h\in \mSigma_h^{\rm KS'}.
\end{array}
\right.
\end{equation}
By \eqref{eq:ndivm+=rbvep1} and \eqref{eq:rdivm+=nbvep1}, on any specific grid, the well-posedness of \eqref{eq:kssaddle} and its equivalence to \eqref{eq:ksred} are evident. 

Now, taking $\vs_h$ to be any one in $\vU^{\vR}_h\subset \vP_1(\mathcal{T}_h)$, we obtain $\dv_h\fzeta_h=\vf_h$, and further, 
$$
(\mathcal{C}\bve_h(\vr_h),\bve_h(\vs_h))+(\fzeta_h,\bve_h(\vs_h))=0,\ \ \forall\,\vs_h\in\vP_1(\mathcal{T}_h),
$$ 
and thus $\mathcal{C}\bve_h(\vr_h)=-\mathbb{P}^0_h\fzeta_h$, namely  $-\bve_h(\vr_h)=\mathcal{A}\mathbb{P}^0_h\fzeta_h=\mathbb{P}^0_h\mathcal{A}\fzeta_h.$ This way, $(\fzeta_h,\mathbb{P}^0_h\vr_h)$ solves the equation \eqref{eq:stressred}. 

On the other hand, if $(\fzeta_h,\bar{\vr}_h)$ solves \eqref{eq:stressred}, by the uniqueness, with the unique solution $\vr_h\in \vV^{\rm KS}_{h0}$ of \eqref{eq:ksred}, $\bve_h(\vr_h)=-\mathcal{A}\mathbb{P}^0_h\fzeta_h$ and $\mathcal{C}\bve_h(\vr_h)=-\mathbb{P}^0_h\fzeta_h$. The proof is completed. 
\end{proof}
\begin{remark}
Given the solution of one of \eqref{eq:ksred} and \eqref{eq:stressred}, the other can be computed cell by cell. The equivalence between the finite element schemes of primal and dual models are also studied in, e.g.,  \cite{Arnold.D;Brezzi.F1985,Marini.L1985sinum,Zhang.S2021SCM}.
\end{remark}

\subsection{Proof of Theorem \ref{thm:wellpose}}

\begin{lemma}\label{lem:inf-sup}
$\displaystyle\inf_{\vv_h\in \vU^{\vR}_h}\sup_{\mtau_h\in \mSigma_h^{\rm KS'}}\frac{(\dv_h\mtau_h,\vv_h)}{\|\mtau_h\|_{\dv,h}\|\vv_h\|_{0,\Omega}}\geqslant C$.
\end{lemma}
\begin{proof}
By the proof of Lemma \ref{lem:equiv}, given $\vv_h\in \vU^{\vR}_h$, there exists a $\mtau_h\in \mSigma_h^{\rm KS'}$, such that $\dv_h\mtau_h=\vv_h$. We can further set $\mtau_h\in \mSigma^{\rm KS',\lrcorner}_h$. By Lemma \ref{lem:pifes}, $\|\mtau_h\|_{0,\Omega}\leqslant C\|\vv_h\|_{0,\Omega}$. The proof of the inf-sup condition then is completed. 
\end{proof}

By Lemma \ref{lem:anotherpi} and that $\R(\dv_h,\mSigma_h^{\rm KS'})=\vU^{\vR}_h$, we can prove the lemma below. 

\begin{lemma}\label{lem:coerc}
Denote $K_h:=\{\mtau_h\in \mSigma_h^{\rm KS'}:\int_\Omega\tr(\mtau_h)=0,\ (\dv\mtau_h,\vv_h)=0,\ \forall\,\vv_h\in \vU^{\vR}_h\}$. Then, with $C$ uniform with respect to $\lambda$, 
\begin{equation}
(\mathcal{A}\mtau_h,\mtau_h)\geqslant C\|\mtau_h\|_{\dv_h}^2:=C(\|\mtau_h\|_{0,\Omega}^2+\|\dv_h\mtau_h\|_{0,\Omega}^2),\ \ \mtau_h\in K_h.
\end{equation}
\end{lemma}

\paragraph{\bf Proof of Theorem \ref{thm:wellpose}}
Denote $\hat\mSigma_h^{\rm KS'}:=\{\mtau_h\in \mSigma_h^{\rm KS'}:\int_\Omega\tr(\mtau_h)=0\}$. Then $\mSigma_h^{\rm KS'}=\hat\mSigma_h^{\rm KS'}\oplus^\perp{\rm span}\{\Id\}$, orthogonal in $\mathbb{L}^2(\Omega,\mathbb{S})$. If $(\msigma_h,\vu_h)$ solves \eqref{eq:fes}, as $\tr(\mathcal{A}\msigma_h)=\frac{1}{2(\lambda+\mu)}\tr(\msigma_h)$ pointwise, $\msigma_h\in \hat\mSigma_h^{\rm KS'}$.

Now, given $\vf$, by Lemmas \ref{lem:inf-sup} and \ref{lem:coerc}, there is a unique solution $(\msigma_h,\vu_h)\in \hat\mSigma_h^{\rm KS'}\times \vU^{\vR}_h$, such that 
\begin{equation}\label{eq:fesp}
\left\{
\begin{array}{rlll}
(\mathcal{A}\msigma_h,\mtau_h)&+(\vu_h,\dv\mtau)&=0&\forall\,\mtau_h\in \hat\mSigma_h^{\rm KS'},
\\
(\dv\msigma_h,\vv_h)&&=(\vf,\vv_h)&\forall\,\vv_h\in \vU^{\vR}_h,
\end{array}
\right.
\end{equation}
and
$$
\|\msigma_h\|_{\dv_h}+\|\vu_h\|_{0,\Omega}\cequiv \|\mathbb{P}^{\vR}_h\vf\|_{0,\Omega}. 
$$
Further, $(\mathcal{A}\msigma_h,\Id)=\frac{1}{2(\lambda+\mu)}\tr(\msigma_h)=0.$ Namely, $(\msigma_h,\vu_h)$ solves \eqref{eq:fes}, and is indeed the unique solution of \eqref{eq:fes}. The proof is completed. 
\qed

\subsection{Proof of Theorem \ref{thm:errest}}

Denote $\vf_h:=\mathbb{P}^{\vR}_h\vf$. Denote by $(\tilde\msigma_h,\tilde\vu_h)$ the solution of \eqref{eq:stressred}, denote by $\tilde\vu_h^{\rm KS}$ the solution of \eqref{eq:ksred}, and let $\vu_h^{\rm KS}\in \vV^{\rm KS}_{h0}$ be such that 
\begin{equation}\label{eq:ksfes}
(\mathcal{C}\bve_h(\vu_h^{\rm KS}),\bve_h(\vs_h))=(\vf,\vs_h),\ \forall\,\vs_h\in \vV^{\rm KS}_{h0}.
\end{equation}

Firstly, we show that
$$
\|\bve_h(\vu_h^{\rm KS})-\bve_h(\tilde\vu_h^{\rm KS})\|_{0,\Omega}+\lambda\|\dv_h\vu_h^{\rm KS}-\dv_h\tilde\vu_h^{\rm KS}\|_{0,\Omega}\leqslant Ch\|\vf\|_{0,\Omega}. 
$$
Indeed, $2\mu\|\bve_h(\vu_h^{\rm KS}-\tilde\vu_h^{\rm KS})\|_{0,\Omega}^2+\lambda\|\dv_h(\vu_h^{\rm KS}-\tilde\vu_h^{\rm KS})\|_{0,\Omega}^2= (\vf,\vu_h^{\rm KS}-\tilde\vu_h^{\rm KS}),$ therefore, $\|\bve_h(\vu_h^{\rm KS}-\tilde\vu_h^{\rm KS})\|_{0,\Omega}\leqslant Ch\|\vf\|_{0,\Omega}$. Then, by Lemma \ref{lem:kornks}, there exists a $\vw_h\in \vV^{\rm KS}_{h0}$, such that $\dv_h\vw_h=\dv_h(\vu_h^{\rm KS}-\tilde\vu_h^{\rm KS})$, and $\|\bve_h(\vw_h)\|_{0,\Omega}\leqslant \|\nabla_h \vw_h\|_{0,\Omega}\leqslant C\|\dv_h(\vu_h^{\rm KS}-\tilde\vu_h^{\rm KS})\|$, and then
$\lambda(\dv(\vu_h^{\rm KS}-\tilde\vu_h^{\rm KS}),\dv\vw_h)=(\bve_h(\vu_h^{\rm KS}-\tilde\vu_h^{\rm KS}),\bve_h(\vw_h))-(\vf,\vw_h-\mathbb{P}^{\vR}_h\vw_h),$ which leads to that $\lambda\|\dv(\vu_h^{\rm KS}-\tilde\vu_h^{\rm KS})\|_{0,\Omega}\leqslant Ch\|\vf\|_{0,\Omega}. $ If follows further then $\|\mathcal{C}(\bve(\vu_h^{\rm KS}-\tilde\vu_h^{\rm KS}))\|_{0,\Omega}\leqslant Ch\|\vf\|_{0,\Omega}$.

Secondly, we show that 
$$
\tilde\msigma_h=\msigma_h,\ \ \mbox{and}\ \  \|\tilde\vu_h-\vu_h\|_{0,\Omega}\leqslant Ch\|\mathcal{A}\msigma_h\|_{0,\Omega}. 
$$
Decompose $\msigma_h:=\msigma_h^\lrcorner+\mathring{\msigma}_h$ with $\mathring{\msigma}_h\in \mathring{\mSigma}^{\rm KS'}_h$ and $\msigma_h^\lrcorner\in \mSigma^{\rm KS',\lrcorner}_h$, and similarly $\tilde\msigma_h=\tilde\msigma_h^\lrcorner+\mathring{\tilde\msigma}_h$. Then 
$$
\mathring{\tilde\msigma}_h=\mathring{\msigma}_h=0,\ \ \mbox{and}\ \ \tilde\msigma_h^\lrcorner=\msigma_h^\lrcorner.
$$
Further, for any $\mtau_h^\lrcorner\in \mSigma^{\rm KS',\lrcorner}_h$, 
$$
(\vu_h-\tilde\vu_h,\dv\mtau_h^\lrcorner)=(\mathcal{A}\msigma_h^\lrcorner,\mtau_h^\lrcorner-\mathbb{P}^0_h\mtau_h^\lrcorner).
$$
Therefore,
$$
\|\vu_h-\tilde\vu_h\|_{0,\Omega}=\sup_{\mtau_h^\lrcorner \in \mSigma^{\rm KS',\lrcorner}_h}\frac{(\mathcal{A}\msigma_h^\lrcorner,\mtau_h^\lrcorner-\mathbb{P}^0_h\mtau_h^\lrcorner)}{\|\dv\mtau_h^\lrcorner\|_{0,\Omega}}\leqslant \|\mathcal{A}\msigma_h\|_{0,\Omega}\sup_{\mtau_h^\lrcorner \in \mSigma^{\rm KS',\lrcorner}_h}\frac{\|\mtau_h^\lrcorner-\mathbb{P}^0_h\mtau_h^\lrcorner\|_{0,\Omega}}{\|\dv\mtau_h^\lrcorner\|_{0,\Omega}}\leqslant Ch\|\mathcal{A}\msigma_h\|\leqslant Ch\|\vf\|_{0,\Omega}.
$$

Therefore, by Lemma \ref{lem:convks},
$$
\|\vu-\vu_h\|_{0,\Omega}\leqslant \|\vu-\vu_h^{\rm KS}\|_{0,\Omega}+\|\vu^{\rm KS}_h-\tilde\vu_h^{\rm KS}\|_{0,\Omega}+\|\tilde{\vu}_h-\tilde\vu_h^{\rm KS}\|_{0,\Omega}+\|\tilde{\vu}_h-\vu_h\|_{0,\Omega}\leqslant Ch(\|\vu\|_{2,\Omega}+\|\vf\|_{0,\Omega}),
$$
$$
\|\msigma-\msigma_h\|_{0,\Omega}\leqslant \|\msigma-\mathcal{C}\bve_h(\vu_h^{\rm KS})\|_{0,\Omega} + \|\mathcal{C}\bve_h(\vu_h^{\rm KS})-\mathcal{C}\bve_h(\tilde\vu_h^{\rm KS})\|_{0,\Omega} + \|\mathcal{C}\bve_h(\tilde\vu_h^{\rm KS})-\msigma_h\|_{0,\Omega}\leqslant Ch(\|\vu\|_{2,\Omega}+\|\vf\|_{0,\Omega}),
$$
and
$$
\|\dv\msigma-\dv\msigma_h\|_{0,\Omega}=\|\dv\msigma-\mathbb{P}_h^{\vR}\dv\msigma\|_{0,\Omega}\leqslant Ch\|\vf\|_{1,\Omega},\ \ \mbox{provided\ that}\ \vf\in \vH^1(\Omega). 
$$
The proof is completed. \qed

\subsection{Basis functions of $\mSigma^{\rm KS'}_h$} 
First of all, the vertices and edges of the triangulation are numbered, such that interior vertices are ahead of boundary vertices, and interior edges are ahead of boundary edges. Now, we present the basis functions of  $\mSigma^{\rm KS'}_h$ by these three steps(Steps 1-3 below).

\paragraph{\bf Step 0} We write $\left\{\vp_i,\ i=1:\#(\mathcal{N}_h);\ \vq_j,\ j=1:\#(\mathcal{E}_h)\right\}$ to be a set of basis functions of $\vV^{\rm KS}$, where $\vp_i=(p_i,0)^\top$ with $p_i$ being a basis function of $V_{h0}$ the linear element space associated with the vertex $a_i$ and $\vq_j=(0,q_j)^\top $ with $q_j$ being a basis function of $V^{\rm CR}_{h0}$ the Crouzeix-Raviart element space associated with the edge $e_j$. Note that the support of $p_i$ is the patch of one vertex $a_i$, and the support of $q_j$ is the patch of one edge $e_j$. Then $\vV^{\rm KS}_{h0}={\rm span}\left\{\vp_i,\ a_i\in\mathcal{N}_h^i;\ \vq_j,\ e_j\in\mathcal{E}_h^i\right\}$.

\paragraph{\bf Step 1} Given a cell $T$ with vertices $a_{i}$, $a_j$ and $a_k$ and edges $e_l$, $e_m$ and $e_n$, we can write
$$
\mSigma^{\rm m+}={\rm span}\left\{\mvarphi_{a_i}^T,\ \mvarphi_{a_j}^T,\ \mvarphi_{a_k}^T,\ \mpsi_{e_l}^T,\ \mpsi_{e_m}^T,\ \mpsi_{e_n}^T\right\}, 
$$
such that 
\begin{multline*}
(\dv\mvarphi_{a_r}^T,\vp_s|_T)_T+(\mvarphi_{a_r}^T,\bve(\vp_s|_T))_T=\delta_{rs},\ r,s\in \left\{i,j,k\right\},
\\ \mbox{and}\ \ (\dv\mvarphi_{a_r}^T,\vq_t|_T)_T+(\mvarphi_{a_r}^T,\bve(\vq_t|_T))_T=0,\ r\in \left\{i,j,k\right\},\ s\in\left\{l,m,n\right\},
\end{multline*}
and
\begin{multline*}
(\dv\mpsi_{e_r}^T,\vp_s|_T)_T+(\mpsi_{e_r}^T,\bve(\vp_s|_T))_T=0,\ r\in \left\{l,m,n\right\},\ s\in\left\{i,j,k\right\},
\\ 
\mbox{and}\ \  (\dv\mvarphi_{a_r}^T,\vq_t|_T)_T+(\mvarphi_{a_r}^T,\bve(\vq_t|_T))_T=\delta_{rt},\ r,t\in \left\{l,m,n\right\}.
\end{multline*}

\paragraph{\bf Step 2} Based on all $T\in\mathcal{T}_h$ and $\left\{\mvarphi_{a_i}^T,\ \mvarphi_{a_j}^T,\ \mvarphi_{a_k}^T,\ \mpsi_{e_l}^T,\ \mpsi_{e_m}^T,\ \mpsi_{e_n}^T\right\}$ associated with $T$, 
\begin{multline}\label{eq:basissigmaks}
\mSigma^{\rm KS'}_h
=\left\{\mtau_h\in \mSigma^{\rm m+}(\mathcal{T}_h):(\dv\mtau_h,\vv_h)+(\mtau_h,\bve_h(\vv_h))=0,\ \forall\,\vv_h\in\vV^{\rm KS}_{h0}\right\}
\\
=\left\{\mtau_h\in \prod_{T\in\mathcal{T}_h}{\rm span}\left\{\mvarphi_{a_i}^T,\ \mvarphi_{a_j}^T,\ \mvarphi_{a_k}^T,\ \mpsi_{e_l}^T,\ \mpsi_{e_m}^T,\ \mpsi_{e_n}^T\right\}:(\dv\mtau_h,\vv_h)+(\mtau_h,\bve_h(\vv_h))=0,\ \forall\,\vv_h\in\vV^{\rm KS}_{h0}\right\}
\\
= \Bigg\{\mtau_h\in \bigoplus_{a_i\in\mathcal{N}_h}\prod_{T:a_i\in\partial T}{\rm span}\left\{\mvarphi_{a_i}^T\right\}\oplus \bigoplus_{e_l\in\mathcal{E}_h}\prod_{T:e_l\subset\partial T}{\rm span}\left\{\mpsi_{e_l}^T\right\}:\qquad\qquad\qquad\qquad\qquad\qquad\qquad
\\
(\dv\mtau_h,\vv_h)+(\mtau_h,\bve_h(\vv_h))=0,\ \forall\,\vv_h\in {\rm span}\left\{\vp_j,\ a_j\in\mathcal{N}_h^i;\ \vq_m,\ e_m\in\mathcal{E}_h^i\right\}\Bigg\}
\\
=\left[\bigoplus_{a_i\in\mathcal{N}_h^b}\prod_{T:a_i\in\partial T}{\rm span}\left\{\mvarphi_{a_i}^T\right\}\right]
\oplus
\left[\bigoplus_{a_j\in\mathcal{N}_h^i} \left\{\mtau_h\in \prod_{T:a_j\in\partial T}{\rm span}\left\{\mvarphi_{a_j}^T\right\}:(\dv\mtau_h,\vp_j)+(\mtau_h,\bve_h(\vp_j))=0\right\}\right]
\\
\oplus \left[\bigoplus_{e_l\in\mathcal{E}_h^b}\prod_{T:e_l\subset\partial T}{\rm span}\left\{\mpsi_{e_l}^T\right\} \right]
\oplus
\left[\bigoplus_{e_m\in\mathcal{E}_h^i} \left\{\mtau_h\in \prod_{T:e_m\subset\partial T}{\rm span}\left\{\mpsi_{e_l}^T\right\}:(\dv\mtau_h,\vq_m)+(\mtau_h,\bve_h(\vq_m))=0\right\}\right];
\end{multline}
here $\mtau_h\in \prod_{T:a_i\in\partial T}{\rm span}\left\{\mvarphi^T_{a_i}\right\}$ means $\mtau_h|_T\in {\rm span}\left\{\mvarphi^T_{a_i}\right\}$ if $a_i\in \partial T$, and otherwise $\mtau_h|_T=0$.

\paragraph{\bf Step 3} The basis functions of $\mSigma^{\rm KS'}_h$ are collected in four groups:

\begin{description}
\item[Group 1] for boundary vertices  $a_i$, all $\mvarphi^i_T$ (with $a_i\in\partial T$) each is a basis function of $\mSigma^{\rm KS'}_h$;
\item[Group 2] for interior vertices  $a_i$, all functions 
$$
\mtau_h\in \prod_{T:a_i\in\partial T}{\rm span}\left\{\mvarphi^T_{a_i}\right\}\ \ \mbox{such\ that} \ \ (\dv\mtau_h,\vp_i)+(\mtau_h,\bve_h(\vp_i))=0;
$$
\item[Group 3] for boundary edges $e_l$, all $\mpsi^l_T$ (with $e_l\subset\partial T$) each is a basis function of $\mSigma^{\rm KS'}_h$;
\item[Group 4] for interior edges $e_l$, all functions 
$$
\mtau_h\in \prod_{T:e_l\subset\partial T}{\rm span}\left\{\mpsi^T_{e_l}\right\}\ \ \mbox{such\ that} \ \ (\dv\mtau_h,\vq_l)+(\mtau_h,\bve_h(\vq_l))=0.
$$
\end{description}

These four groups of basis functions are linearly independent, and they span the whole space $\mSigma^{\rm KS'}_h$. They are corresponding to the last four terms of \eqref{eq:basissigmaks} respectively. 
~\\

\paragraph{\bf Supports of the basis functions} Firstly, it is easy to know the basis functions in Groups 1 and 3 are each supported on a single cell. 

Then, for the fourth group, as an interior edge is shared by two cells, the support of the basis function is a patch by two neighbored cells. 

For the second group, if an interior vertex $a_i$ is shared by $m$ cells, then there are exactly $m-1$ linear independent basis functions in Group 2 associated with $a_i$, and the supports are each a patch by two neighboured cells. For example, c.f. Figure \ref{fig:basisinteriorvertex}, an interior vertex $A$ is shared by 6 cells $T_i$, $i=1:6$. Associated with $\vp_A$, there are 6 local functions $\mvarphi^{T_i}_A$ each on one cell. The associated basis function $\mtau_h$ is piecewise $c_i\mvarphi^{T_i}_A$, and by the criterion $\sum_{i=1}^6(\dv c_i\mvarphi^{T_i}_A,\vp_r)_{T_i}+(c_i\mvarphi^{T_i}_A,\bve(\vp_A))_{T_i}=0,$ 5 linearly independent basis functions can be determined, the supports of which are a patch by two neighbored cells, illustrated in Figure \ref{fig:basisinteriorvertex}. Note that we can choose any 5 favored patches by neighbored cells of the 6 patches to be the supports of the linearly independent basis functions.

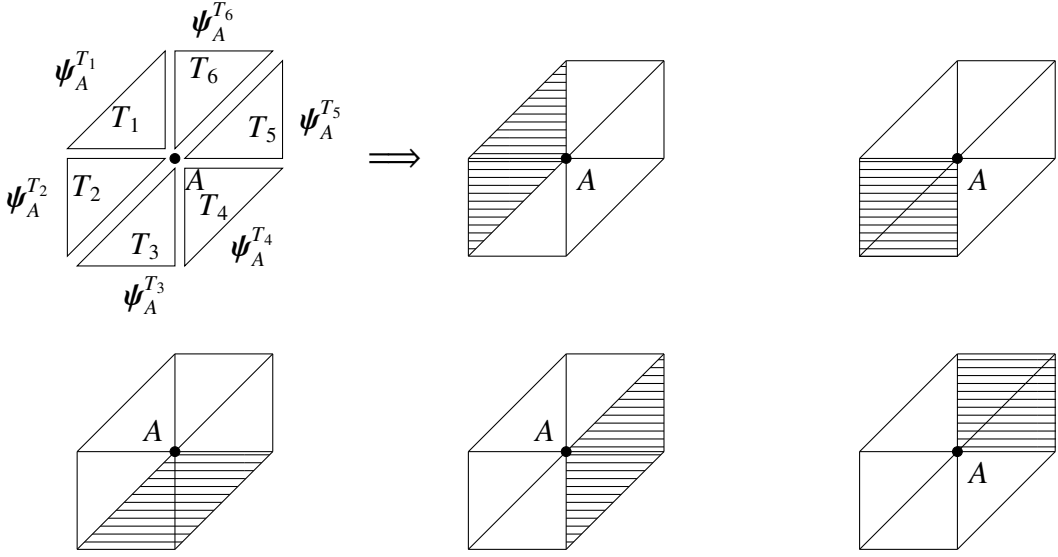
\begin{figure}
\begin{tikzpicture}[scale=1.3]
\path 	
coordinate (a1mb0) at (0.9,0)
coordinate (a1pb0m) at (1.1,-0.1)
coordinate (a1pb0p) at (1.1,0.1)
coordinate (a1mb0m) at (0.9,-0.1)
coordinate (a1mb0p) at (0.9,0.1)
coordinate (a2pb0p) at (2.1,0.1)
coordinate (a2pb0m) at (2.1,-0.1)
coordinate (a2mb0p) at (1.9,0.1)
coordinate (a2mb0m) at (1.9,-0.1)
coordinate (a1b1m) at (1,0.9)
coordinate (a1mb1) at (0.9,1)
coordinate (a1pb1p) at (1.1,1.1)
coordinate (a1pb1m) at (1.1,0.9)
coordinate (a1mb1p) at (0.9,1.1)
coordinate (a1mb1m) at (0.9,0.9)
coordinate (a2b1m) at (2,0.9)
coordinate (a2pb1p) at (2.1,1.1)
coordinate (a2pb1m) at (2.1,0.9)
coordinate (a2mb1p) at (1.9,1.1)
coordinate (a2mb1m) at (1.9,0.9)
coordinate (a3pb1p) at (3.1,1.1)
coordinate (a3pb1m) at (3.1,0.9)
coordinate (a3mb1p) at (2.9,1.1)
coordinate (a3mb1m) at (2.9,0.9)
coordinate (a1mb2) at (0.9,2)
coordinate (a1pb2p) at (1.1,2.1)
coordinate (a1pb2m) at (1.1,1.9)
coordinate (a1mb2p) at (0.9,2.1)
coordinate (a1mb2m) at (0.9,1.9)
coordinate (a2pb2) at (2.1,2)
coordinate (a2b2p) at (2,2.1)
coordinate (a2b2m) at (2,1.9)
coordinate (a2mb2) at (1.9,2)
coordinate (a2pb2p) at (2.1,2.1)
coordinate (a2pb2m) at (2.1,1.9)
coordinate (a2mb2p) at (1.9,2.1)
coordinate (a2mb2m) at (1.9,1.9)
coordinate (a3pb2) at (3.1,2)
coordinate (a3pb2p) at (3.1,2.1)
coordinate (a3pb2m) at (3.1,1.9)
coordinate (a3mb2p) at (2.9,2.1)
coordinate (a3mb2m) at (2.9,1.9)
coordinate (a2b3p) at (2,3.1)
coordinate (a2pb3p) at (2.1,3.1)
coordinate (a2pb3m) at (2.1,2.9)
coordinate (a2mb3p) at (1.9,3.1)
coordinate (a2mb3m) at (1.9,2.9)
coordinate (a3pb3) at (3.1,3)
coordinate (a3b3p) at (3,3.1)
coordinate (a3pb3p) at (3.1,3.1)
coordinate (a3pb3m) at (3.1,2.9)
coordinate (a3mb3p) at (2.9,3.1)
coordinate (a3mb3m) at (2.9,2.9);

\draw[line width=.4pt]  (a1mb1) -- (a1mb2) -- (a2mb2) -- cycle;
\draw[line width=.4pt]  (a1b1m) -- (a2b2m) -- (a2b1m) -- cycle;
\draw[line width=.4pt]  (a1mb2p) -- (a2mb2p) -- (a2mb3p) -- cycle;
\draw[line width=.4pt]  (a2b2p) -- (a2b3p) -- (a3b3p) -- cycle;
\draw[line width=.4pt]  (a2pb2) -- (a3pb2) -- (a3pb3) -- cycle;
\draw[line width=.4pt]  (a2pb2m) -- (a3pb2m) -- (a2pb1m) -- cycle;

\node at(1.5,2.4){$T_1$};

\node at(1,2.9){$\mpsi^{T_1}_A$};

\node at(2.3,2.9) {$T_6$};

\node at(2.4,3.4){$\mpsi_A^{T_6}$};

\node at(2.9,2.3) {$T_5$};
\node at(3.5,2.4){$\mpsi_A^{T_5}$};

\node at(2.4,1.5){$T_4$};
\node at(2.8,1.1){$\mpsi_A^{T_4}$};

\node at(1.7,1.1){$T_3$};
\node at(1.7,0.6){$\mpsi_A^{T_3}$};

\node at(1.1,1.7){$T_2$};
\node at(0.5,1.6){$\mpsi_A^{T_2}$};

\node at(4.25,2){\large$\Longrightarrow$};

\draw[fill] (2,2) circle [radius=0.05];
\node[below right] at (2,2) {$A$};


\path 	coordinate (aa2b1) at (5,1)
coordinate (aa3b1) at (6,1)
coordinate (aa4b2) at (7,2)
coordinate (aa4b3) at (7,3)
coordinate (aa3b3) at (6,3)
coordinate (aa2b2) at (5,2)
coordinate (aa3b2) at (6,2);

\draw[line width=.4pt]  (aa2b2) -- (aa3b3);
\draw[line width=.4pt]  (aa2b1) -- (aa4b3);
\draw[line width=.4pt]  (aa3b1) -- (aa4b2);
\draw[line width=.4pt]  (aa3b3) -- (aa4b3);
\draw[line width=.4pt]  (aa2b2) -- (aa4b2);
\draw[line width=.4pt]  (aa2b1) -- (aa3b1);
\draw[line width=.4pt]  (aa2b1) -- (aa2b2);
\draw[line width=.4pt]  (aa3b1) -- (aa3b3);
\draw[line width=.4pt]  (aa4b2) -- (aa4b3);

\path 	coordinate (ba2b1) at (9,1)
coordinate (ba3b1) at (10,1)
coordinate (ba4b2) at (11,2)
coordinate (ba4b3) at (11,3)
coordinate (ba3b3) at (10,3)
coordinate (ba2b2) at (9,2)
coordinate (ba3b2) at (10,2);

\draw[line width=.4pt]  (ba2b2) -- (ba3b3);
\draw[line width=.4pt]  (ba2b1) -- (ba4b3);
\draw[line width=.4pt]  (ba3b1) -- (ba4b2);
\draw[line width=.4pt]  (ba3b3) -- (ba4b3);
\draw[line width=.4pt]  (ba2b2) -- (ba4b2);
\draw[line width=.4pt]  (ba2b1) -- (ba3b1);
\draw[line width=.4pt]  (ba2b1) -- (ba2b2);
\draw[line width=.4pt]  (ba3b1) -- (ba3b3);
\draw[line width=.4pt]  (ba4b2) -- (ba4b3);


\path 	coordinate (ca2b1) at (1,-2)
coordinate (ca3b1) at (2,-2)
coordinate (ca4b2) at (3,-1)
coordinate (ca4b3) at (3,0)
coordinate (ca3b3) at (2,0)
coordinate (ca2b2) at (1,-1)
coordinate (ca3b2) at (2,-1);

\draw[line width=.4pt]  (ca2b2) -- (ca3b3);
\draw[line width=.4pt]  (ca2b1) -- (ca4b3);
\draw[line width=.4pt]  (ca3b1) -- (ca4b2);
\draw[line width=.4pt]  (ca3b3) -- (ca4b3);
\draw[line width=.4pt]  (ca2b2) -- (ca4b2);
\draw[line width=.4pt]  (ca2b1) -- (ca3b1);
\draw[line width=.4pt]  (ca2b1) -- (ca2b2);
\draw[line width=.4pt]  (ca3b1) -- (ca3b3);
\draw[line width=.4pt]  (ca4b2) -- (ca4b3);

\path 	coordinate (da2b1) at (5,-2)
coordinate (da3b1) at (6,-2)
coordinate (da4b2) at (7,-1)
coordinate (da4b3) at (7,0)
coordinate (da3b3) at (6,0)
coordinate (da2b2) at (5,-1)
coordinate (da3b2) at (6,-1);

\draw[line width=.4pt]  (da2b2) -- (da3b3);
\draw[line width=.4pt]  (da2b1) -- (da4b3);
\draw[line width=.4pt]  (da3b1) -- (da4b2);
\draw[line width=.4pt]  (da3b3) -- (da4b3);
\draw[line width=.4pt]  (da2b2) -- (da4b2);
\draw[line width=.4pt]  (da2b1) -- (da3b1);
\draw[line width=.4pt]  (da2b1) -- (da2b2);
\draw[line width=.4pt]  (da3b1) -- (da3b3);
\draw[line width=.4pt]  (da4b2) -- (da4b3);

\path 	coordinate (ea2b1) at (9,-2)
coordinate (ea3b1) at (10,-2)
coordinate (ea4b2) at (11,-1)
coordinate (ea4b3) at (11,0)
coordinate (ea3b3) at (10,0)
coordinate (ea2b2) at (9,-1)
coordinate (ea3b2) at (10,-1);

\draw[line width=.4pt]  (ea2b2) -- (ea3b3);
\draw[line width=.4pt]  (ea2b1) -- (ea4b3);
\draw[line width=.4pt]  (ea3b1) -- (ea4b2);
\draw[line width=.4pt]  (ea3b3) -- (ea4b3);
\draw[line width=.4pt]  (ea2b2) -- (ea4b2);
\draw[line width=.4pt]  (ea2b1) -- (ea3b1);
\draw[line width=.4pt]  (ea2b1) -- (ea2b2);
\draw[line width=.4pt]  (ea3b1) -- (ea3b3);
\draw[line width=.4pt]  (ea4b2) -- (ea4b3);

\draw[fill] (aa3b2) circle [radius=0.05];
\node[below right] at (aa3b2) {$A$};
\draw[fill] (ba3b2) circle [radius=0.05];
\node[below right] at (ba3b2) {$A$};
\draw[fill] (ca3b2) circle [radius=0.05];
\node[above left] at (ca3b2) {$A$};
\draw[fill] (da3b2) circle [radius=0.05];
\node[above left] at (da3b2) {$A$};
\draw[fill] (ea3b2) circle [radius=0.05];
\node[below right] at (ea3b2) {$A$};

\begin{scope}
\fill[pattern=horizontal lines] (aa2b1)--(aa3b2)--(aa3b3)--(aa2b2)--(aa2b1);
\fill[pattern=horizontal lines] (ba2b1)--(ba3b1)--(ba3b2)--(ba2b2)--(ba2b1);
\fill[pattern=horizontal lines] (ca2b1)--(ca3b1)--(ca4b2)--(ca3b2)--(ca2b1);
\fill[pattern=horizontal lines] (da3b1)--(da4b2)--(da4b3)--(da3b2)--(da3b1);
\fill[pattern=horizontal lines] (ea3b2)--(ea4b2)--(ea4b3)--(ea3b3)--(ea3b2);
\end{scope}

\end{tikzpicture}

\caption{As $A$ is shared by six triangles, five basis functions are associated with the interior vertex $A$. The shadowed parts are respectively the supports of the basis functions.}\label{fig:basisinteriorvertex}
\end{figure}
~\\

\paragraph{\bf Size of stiffness matrices} Given a triangulation, the dimension of $\mSigma^{\rm KS'}_h$ is $6\#(\mathcal{T}_h)-(\#(\mathcal{E}_h^i)+\#(\mathcal{N}_h^i))$, which is asymptotically $8$ times of $\#(\mathcal{N}_h^i)$. The size of the global stiffness matrix is $(9\#(\mathcal{T}_h)-(\#(\mathcal{E}_h^i)+\#(\mathcal{N}_h^i)))\times (9\#(\mathcal{T}_h)-(\#(\mathcal{E}_h^i)+\#(\mathcal{N}_h^i)))$. A cell $T$ is covered by the supports of no more than 9 basis functions of $\mSigma^{\rm KS'}_h$, which are:
\begin{itemize}
\item 3 in Group 4, and no more than 6 in Group 2, if $T$ is an interior cell;
\item 1 in Group 1, 3 in Group 4, and no more than 4 in Group 2, if $T$ has 1 boundary vertex;
\item 2 in Group 1, 1 in Group 3, 2 in Group 4, and no more than 2 in Group 2, if $T$ has 1 boundary edge. 
\end{itemize}
Thus, a cell $T$ is covered by the supports of no more than 12 basis functions, including 3 of $\vU^{\vR}_h$. Therefore, the size of the cell-wise stiffness matrix is not bigger than $12\times 12$.

\section{Investigation on finite element schemes of theoretically lowest degrees}

\label{sec:lowest}

\subsection{A lowest-degree robust Hellinger-Reissner finite element scheme}

We consider the finite element problem: find $(\msigma^{\rm m}_h,\vu^{\rm m}_h)\in \mSigma^{\rm rKS'}_h\times \vU^{\rm m}_h:=\vP_0(\mathcal{T}_h)$, such that 
\begin{equation}\label{eq:feslsto}
\left\{
\begin{array}{rlll}
(\mathcal{A}\msigma^{\rm m}_h,\mtau_h)&+(\vu_h^{\rm m},\dv_h\mtau_h)&=0&\forall\,\mtau_h\in \mSigma^{\rm rKS'}_h,
\\
(\dv_h\msigma_h^{\rm m},\vv_h)&&=(\vf,\vv_h)&\forall\,\vv_h\in \vU^{\rm m}_h. 
\end{array}
\right.
\end{equation}

The wellposedness of \eqref{eq:feslsto} is straightforward to obtain. 

\begin{theorem}
Let $(\msigma,\vu)$ and $(\msigma_h^{\rm m},\vu^{\rm m}_h)$ be the solutions of \eqref{eq:feslsto}, respectively. Assume $\vf\in\vH^1(\Omega)$, and $\vu\in\vH^2(\Omega)$. Then
$$
\|\dv_h\msigma_h^{\rm m}-\dv\msigma\|_{0,\Omega}\leqslant Ch\|\vf\|_{1,\Omega},
$$ 
$$
\|\msigma_h^{\rm m}-\msigma\|_{0,\Omega}\leqslant Ch(\|\vu\|_{2,\Omega}+\|\vf\|_{1,\Omega}),
$$
and
$$
\|\vu^m_h-\vu\|_{0,\Omega}\leqslant  Ch(\|\vu\|_{2,\Omega}+\|\vf\|_{1,\Omega}).
$$
\end{theorem}
\begin{proof}
Firstly, similar to the proof of Theorem \ref{thm:errest}, we can show $\msigma_h^{\rm m}\in \mSigma^{\rm rKS',\lrcorner}_h$. Note that $(\msigma_h,\vu_h)$, the solution of \eqref{eq:fes}, satisfies \eqref{eq:feslsto}. It follows then 
$$
\|\msigma_h^{\rm m}-\msigma_h\|_{0,\Omega}\leqslant C\|\dv_h(\msigma_h^{\rm m}-\msigma_h)\|\ \ \mbox{and}\ \ \|\mathbb{P}^0_h\vu_h-\vu^{\rm m}_h\|_{0,\Omega}\leqslant C\|\msigma^{\rm m}_h-\msigma_h\|_{0,\Omega}.
$$
Therefore, provided $\vf\in\vH^1(\Omega)$, 
$$
\|\dv_h\msigma_h^{\rm m}-\dv\msigma\|_{0,\Omega}\leqslant Ch\|\vf\|_{1,\Omega},
$$ 
$$
\|\msigma_h^{\rm m}-\msigma\|_{0,\Omega}\leqslant Ch(\|\vu\|_{2,\Omega}+\|\vf\|_{1,\Omega}),
$$
and
$$
\|\vu^{\rm m}_h-\vu\|_{0,\Omega}\leqslant \|\mathbb{P}^0_h\vu_h-\vu^0_h\|_{0,\Omega}+\|\mathbb{P}^0_h\vu_h-\mathbb{P}^0_h\vu\|_{0,\Omega}+\|\mathbb{P}^0_h\vu-\vu\|_{0,\Omega}\leqslant Ch(\|\vu\|_{2,\Omega}+\|\vf\|_{1,\Omega}).
$$
The proof is completed.  
\end{proof}

\subsection{A lowest-degree robust Navier-Lam\'e finite element scheme}

For the Navier-Lam\'e equation, we consider the finite element problem: find $\vu_h^{\bve}\in \vV^{\rm rKS}_{h0}$, such that 
\begin{equation}\label{eq:minimalNL}
(\mathcal{C}\bve_h(\vu_h^{\bve}),\bve_h(\vv_h))=(\vf,\mathbb{P}^0_h\vv_h),\ \ \forall\,\vv_h\in \vV^{\rm rKS}_{h0}. 
\end{equation}

\begin{theorem}\label{thm:convmNL}
Let $\vu$ and $\vu^{\bve}_h$ be the solutions of \eqref{eq:primalmodel} and \eqref{eq:minimalNL}, respectively. Assume $\vf\in \vH^1(\Omega)$ and $\vu\in \vH^2(\Omega)$. Then 
$$
\|\bve_h(\vu^{\bve}_h)-\bve(\vu)\|_{0,\Omega}+\|\mathcal{C}(\bve_h(\vu^{\bve}_h)-\bve(\vu))\|_{0,\Omega}\leqslant Ch(\|\vu\|_{2,\Omega}+\|\vf\|_{1,\Omega}).
$$
\end{theorem}
\begin{proof}
We firstly introduce an auxiliary problem: find $(\tilde{\msigma}_h,\tilde{\vu}_h)\in \mSigma^{\rm rKS'}_h\times \vU^{\rm m}_h$, such that 
\begin{equation}
\left\{
\begin{array}{rlll}
(\mathcal{A}\mathbb{P}^0_h\tilde\msigma_h,\mathbb{P}^0_h\mtau_h)&+(\tilde\vu_h,\dv\mtau)&=0&\forall\,\mtau_h\in \mSigma^{\rm rKS'}_h,
\\
(\dv\tilde\msigma_h,\vv_h)&&=(\vf,\vv_h)&\forall\,\vv_h\in \vU^{\rm m}_h. 
\end{array}
\right.
\end{equation}
Then by the virtue of Lemma \ref{lem:equiv}, we can show that $\mathcal{C}\bve_h(\vu_h^{\bve})=\mathbb{P}^0_h\tilde\msigma_h$ and $\tilde{\vu}_h=\mathbb{P}^0_h\vu^{\bve}_h$.

At the same time,
\begin{equation}
\left\{
\begin{array}{rlll}
(\mathcal{A}\mathbb{P}^0_h(\tilde\msigma_h-\msigma^{\rm m}_h),\mtau_h)&+(\tilde\vu_h-\vu_h^{\rm m},\dv\mtau)&=(\mathcal{A}(\msigma^{\rm m}_h-\mathbb{P}^0_h\msigma^{\rm m}_h),\mtau_h)&\forall\,\mtau_h\in \mSigma^{\rm rKS'}_h,
\\
(\dv\tilde\msigma_h-\msigma^{\rm m}_h,\vv_h)&&=0&\forall\,\vv_h\in \vU^{\rm 0}_h. 
\end{array}
\right.
\end{equation}
It follows that
$$
\|\tilde\msigma_h-\msigma^{\rm m}_h\|_{\dv_h}+\|\tilde\vu_h-\vu_h^{\rm m}\|_{0,\Omega}\cequiv \|\mathcal{A}(\msigma^{\rm m}_h-\mathbb{P}^0_h\msigma^{\rm m}_h)\|_{0,\Omega}\leqslant Ch\|\vf\|_{0,\Omega}.
$$
Therefore,
$$
\|\vu-\vu^{\bve}_h\|_{0,\Omega}\leqslant \|\vu-\tilde{\vu}_h\|_{0,\Omega}+\|\vu^{\bve}_h-\mathbb{P}^0_h\vu^{\bve}_h\|_{0,\Omega}\leqslant \|\vu-\tilde{\vu}_h\|_{0,\Omega}+Ch\|\bve_h(\vu^{\bve}_h)\|_{0,\Omega}\leqslant Ch\|\vf\|_{1,\Omega},
$$
$$
\|\bve_h(\vu^{\bve}_h)-\bve(\vu)\|_{0,\Omega}=\|\mathbb{P}^0_h\mathcal{A}\tilde{\msigma}_h-\mathcal{A}\msigma\|_{0,\Omega}\leqslant C(\|\vu\|_{2,\Omega}+\|\vf\|_{1,\Omega}),
$$
and
$$
\|\mathcal{C}(\bve_h(\vu^{\bve}_h)-\bve(\vu))\|_{0,\Omega}=\|\mathbb{P}^0_h\tilde{\msigma}_h-\msigma\|_{0,\Omega}\leqslant C(\|\vu\|_{2,\Omega}+\|\vf\|_{1,\Omega}).
$$
The proof is completed. 
\end{proof}
\begin{remark}
Theorem \ref{thm:convmNL} presents the convergence of the scheme in broken energy norm. We remark here that, as the shape function space $\vV^{\bve,\rm m}$ does not contain $\left(\begin{array}{c}y\\ -x\end{array}\right)$, the convergence of $\|\nabla_h(\vu-\vu^{\bve}_h)\|_{0,\Omega}$ may not in general be expected. 
\end{remark}

\begin{remark}
In this paper, we do not seek to figure out the basis functions of $\mSigma^{\rm rKS'}_h$ or $\vV^{\rm rKS}_{h0}$. Though, by the virtue of \eqref{eq:kssaddle}, we can write \eqref{eq:feslsto} and \eqref{eq:minimalNL} as expanded formulation by introducing Lagrangian multipliers, and to have the scheme implemented. 
\end{remark}

\end{document}